\documentclass[11pt,reqno,tbtags]{amsart}
\usepackage{amssymb}
\usepackage{natbib}
\bibpunct[, ]{[}{]}{;}{a}{,}{,}
\numberwithin{equation}{section}
\allowdisplaybreaks

\newtheorem{theorem}{Theorem}[section]
\newtheorem{lemma}[theorem]{Lemma}
\newtheorem{proposition}[theorem]{Proposition}
\newtheorem{corollary}[theorem]{Corollary}

\theoremstyle{definition}

\newtheorem{remark}[theorem]{Remark}

\newtheorem*{acks}{Acknowledgements}

\theoremstyle{remark}

\newcounter{thmenumerate}
\newenvironment{thmenumerate}
{\setcounter{thmenumerate}{0}%
 \def\item{\par
 \refstepcounter{thmenumerate}\textup{(\roman{thmenumerate})\enspace}}
}
{}

\newcounter{xenumerate}   

\newcommand\pfitem[1]{\par(#1):}

\newcounter{CC} 
\newcommand{\CC}{\stepcounter{CC}\CCx} 
\newcommand{\CCx}{C_{\arabic{CC}}}     
\newcounter{cc}
\newcommand{\cc}{\stepcounter{cc}\ccx} 
\newcommand{\ccx}{c_{\arabic{cc}}}     

\newcommand{\refT}[1]{Theorem~\ref{#1}}

\newcommand{\refL}[1]{Lemma~\ref{#1}}
\newcommand{\refR}[1]{Remark~\ref{#1}}
\newcommand{\refS}[1]{Section~\ref{#1}}
\newcommand{\refA}[1]{Appendix~\ref{#1}}

\newcommand{\refand}[2]{\ref{#1} and~\ref{#2}}

\begingroup
  \divide\count255 by 60
  \multiply\count255 by -60
  \advance\count255 by \time
  \ifnum \count255 < 10 \xdef\klockan{\the\count1.0\the\count255}
  \else\xdef\klockan{\the\count1.\the\count255}\fi
\endgroup

\newcommand\set[1]{\ensuremath{\{#1\}}}
\newcommand\bigset[1]{\ensuremath{\bigl\{#1\bigr\}}}

\newcommand\xpar[1]{(#1)}
\newcommand\bigpar[1]{\bigl(#1\bigr)}
\newcommand\Bigpar[1]{\Bigl(#1\Bigr)}
\newcommand\biggpar[1]{\biggl(#1\biggr)}
\newcommand\lrpar[1]{\left(#1\right)}

\newcommand\bigabs[1]{\bigl|#1\bigr|}

\newcommand\lrabs[1]{\left|#1\right|}
\def\rompar(#1){\textup(#1\textup)}    
\newcommand\xfrac[2]{#1/#2}

\newcommand\parfrac[2]{\Bigpar{\frac{#1}{#2}}}

\newcommand\expx[1]{\exp\bigl(#1\bigr)}
\newcommand\expQ[1]{e^{#1}}
\def\xexp(#1){e^{#1}}

\newcommand\ntoo{\ensuremath{{n\to\infty}}}

\newcommand\rtoo{\ensuremath{{r\to\infty}}}
\newcommand\xtoo{\ensuremath{{x\to\infty}}}
\newcommand\ztoo{\ensuremath{{|z|\to\infty}}}
\newcommand\ttoo{\ensuremath{{t\to\infty}}}

\newcommand\half{\tfrac12}

\newcommand\ie{i.e.\spacefactor=1000}
\newcommand\eg{e.g.\spacefactor=1000}
\newcommand\viz{viz.\spacefactor=1000}
\newcommand\cf{cf.\spacefactor=1000}
\newcommand{\as}{a.s.\spacefactor=1000}

\newcommand\ii{\mathbf{i}}

\newcommand\eqd{\overset{\mathrm{d}}{=}}

\newcommand\bbR{\mathbb R}
\newcommand\bbC{\mathbb C}

\renewcommand\Re{\operatorname{Re}}

\newcommand\E{\operatorname{\mathbb E{}}}
\renewcommand\P{\operatorname{\mathbb P{}}}

\newcommand\ga{\alpha}
\newcommand\gb{\beta}
\newcommand\gd{\delta}

\newcommand\gam{\gamma}

\newcommand\gl{\lambda}

\newcommand\gs{\sigma}

\newcommand\gt{\theta}
\newcommand\eps{\varepsilon}

\renewcommand\phi{\varphi}

\newcommand\tg{\tilde g}
\newcommand\fh{h}
\newcommand\tfh{\tilde h}
\newcommand\hfh{\hat h}

\newcommand\cB{\mathcal B}

\def\[#1]{[\![#1]\!]}

\newcommand\qq{^{1/2}}
\newcommand\qqq{^{1/3}}
\newcommand\qqqq{^{1/4}}
\newcommand\qqa{^{2/3}}
\newcommand\qqc{^{3/2}}
\newcommand\qqw{^{-1/2}}
\newcommand\qqqw{^{-1/3}}
\newcommand\qqqqw{^{-1/4}}
\newcommand\qqaw{^{-2/3}}
\newcommand\qqcw{^{-3/2}}
\newcommand\qw{^{-1}}
\newcommand\qww{^{-2}}
\newcommand\sqrtt{\sqrt{2}\,}

\renewcommand{\=}{:=}

\newcommand\intoi{\int_0^1}

\newcommand\intoo{\int_0^\infty}
\newcommand\intoooo{\int_{-\infty}^\infty}
\newcommand\oi{[0,1]}

\newcommand\dd{\,\mathrm{d}}

\newcommand\rhs{right hand side}

\newcommand\br{_{\mathrm{br}}}
\newcommand\bm{_{\mathrm{bm}}}
\newcommand\dm{_{\mathrm{dm}}}
\newcommand\me{_{\mathrm{me}}}
\newcommand\ex{_{\mathrm{ex}}}
\newcommand\brp{_{\mathrm{br+}}}
\newcommand\bmp{_{\mathrm{bm+}}}
\newcommand\brm{_{\mathrm{br-}}}

\newcommand\bmpm{_{\mathrm{bm\pm}}}
\newcommand\cBex{\cB\ex}
\newcommand\cBbr{\cB\br}
\newcommand\cBbm{\cB\bm}
\newcommand\cBme{\cB\me}
\newcommand\cBdm{\cB\dm}
\newcommand\cBbrp{\cB\brp}
\newcommand\cBbmp{\cB\bmp}
\newcommand\xbr{{\mathrm{br}}}
\newcommand\xbm{{\mathrm{bm}}}
\newcommand\xme{{\mathrm{me}}}
\newcommand\xdm{{\mathrm{dm}}}
\newcommand\xex{{\mathrm{ex}}}
\newcommand\xbrp{{\mathrm{br+}}}
\newcommand\xbmp{{\mathrm{bm+}}}

\newcommand\Ai{\mathrm{Ai}}
\newcommand\AI{\mathrm{AI}} 
\newcommand\Bi{\mathrm{Bi}}
\newcommand\BI{\mathrm{BI}}

\newcommand\phit{t}
\newcommand\phif{F}
\newcommand\tphif{\widetilde F}
\newcommand\glx{x}
\newcommand\glz{z}
\newcommand\Hx{\Psi^*}
\newcommand\Hxbrp{\Hx\brp}
\newcommand\Hxbmp{\Hx\bmp}
\newcommand\zetax{{2z^{3/2}/3}}

\newcommand\xxi{_{\Xi}}
\newcommand\xxix[1]{_{#1}^{\Xi}}
\newcommand\gamxi{\gam}
\newcommand\aaz{|\arg z|}
\newcommand\hh{h^0}
\newcommand\hhh{h^1}
\newcommand\fx{f^*}
\newcommand\Qc{Q^c}
\newcommand\slitplane{the slit plane $\bbC\setminus(-\infty,0]$}
\newcommand\xu{v}
\newcommand\tu{\tilde u}
\newcommand\tv{\tilde v}
\newcommand\tgt{\tilde \gt}
\newcommand\ttet{\tilde \tet}
\newcommand\qr{w}
\newcommand\xr{\tu}
\newcommand\xs{\ttet}

\newcommand{\Takacs}{Tak\'acs}

\newcommand{\Voblyi}{Vobly\u\i}
\newcommand\PW{Perman and Wellner \cite{PermanW}}

\newcommand{\maple}{\texttt{Maple}}

\newcommand\REM[1]{\texttt{[#1]}\marginal{XXX}}

\newcommand\urladdrx[1]{\urladdr{\def~{\~{}}#1}}

\newtheorem{thm}{Theorem}[section]

\newtheorem{prop}[thm]{Proposition}
\newtheorem{deff}[thm]{Definition}
\newtheorem{conj}[thm]{Conjecture}
\newcommand{\bconj}{\begin{conj}}
\newcommand{\econj}{\end{conj}}
\newcommand{\bth}{\begin{thm}}
\newcommand{\ethGL}{\end{thm}}
\newcommand{\el}{\end{lem}}
\newcommand{\bdf}{\begin{deff}}
\newcommand{\edf}{\end{deff}}
\newcommand{\bcor}{\begin{cor}}
\newcommand{\ecor}{\end{cor}}
\newcommand{\bprop}{\begin{prop}}
\newcommand{\eprop}{\end{prop}}
\newcommand{\brem}{\begin{rem}}
\newcommand{\erem}{\end{rem}}
\newcommand{\beq}{\begin{equation}}
\newcommand{\eeq}{\end{equation}}
\newcommand{\beqs}{\begin{equation*}}
\newcommand{\eeqs}{\end{equation*}}
\newcommand{\beqn}{\begin{eqnarray}}
\newcommand{\eeqn}{\end{eqnarray}}
\newcommand{\beqns}{\begin{eqnarray*}}
\newcommand{\eeqns}{\end{eqnarray*}}

\newcommand{\ba}{\begin{array}}
\newcommand{\ea}{\end{array}}
\newcommand{\bit}{\begin{itemize}}
\newcommand{\eit}{\end{itemize}}
\newcommand{\ben}{\begin{enumerate}}
\newcommand{\een}{\end{enumerate}}
\newcommand{\bal}{\begin{align}}
\newcommand{\bals}{\begin{align*}}
\newcommand{\bs}{\begin{skip}}
\newcommand{\eal}{\end{align}}
\newcommand{\eals}{\end{align*}}
\newcommand{\es}{\end{skip}}
\newcommand{\babs}{\begin{abstract}}
\newcommand{\eabs}{\end{abstract}}

\newcommand{\lp}{\left (}
\newcommand{\rp}{\right )}
\newcommand{\lb}{\left [}
\newcommand{\rb}{\right ]}

\def\ii{\mathbf{i}}

\newcommand{\al}{\alpha}
\newcommand{\be}{\beta}

\newcommand{\Fi}{\varphi}

\newcommand{\II}{\infty}
\newcommand{\tet}{\theta}

\def\P{{\mathbb {P}}}

\begin{document}
\title[Tail estimates for Brownian areas]
{Tail estimates for the Brownian excursion area and other Brownian areas}

\date{July 4, 2007} 

\author{Svante Janson}
\address{Department of Mathematics, Uppsala University, PO Box 480,
SE-751~06 Uppsala, Sweden}
\email{svante.janson@math.uu.se}
\urladdrx{http://www.math.uu.se/~svante/}

\author{Guy Louchard}
\address{Universit\'e Libre de Bruxelles,
D\'epartement d'Informatique, CP 212, Boulevard du Triomphe, B-1050
Bruxelles, Belgium} 
\email{louchard@ulb.ac.be}
\urladdr{http://www.ulb.ac.be/di/mcs/louchard/}


\begin{abstract}
Several Brownian areas are considered in this paper: the Brownian
excursion area, the Brownian bridge area, the Brownian motion area,
the Brownian meander area, the Brownian double meander area,  the
positive part of Brownian bridge area,  the positive part of Brownian
motion area. 
We are interested in the asymptotics of the right tail of their
density function. Inverting a double Laplace transform, we can derive,
in a mechanical way, all terms of an asymptotic expansion. We
illustrate our technique with the computation of the first four
terms. We also obtain asymptotics for the right tail of the
distribution function and for the moments. Our main tool is the
two-dimensional saddle point method. 
 \end{abstract}

\maketitle

\section{Introduction}\label{Sintro}

Let $B\ex(t)$, $t\in\oi$, be a (normalized) Brownian excursion, and let
$\cBex\=\intoi B\ex(t)\dd t$ be its area (integral). 
This random variable has been studied by several authors, including
Louchard \cite{Lou:kac,Lou:ex},
\Takacs{} \cite{Takacs:Bernoulli}, and
Flajolet and Louchard \cite{FL:Airy}; see also the survey 
by Janson \cite{SJ201} with many further references.

It is known that $\cBex$ has a density function $f\ex$, which was
given explicitly 
by \Takacs{} \cite{Takacs:Bernoulli} 
as a convergent series involving a confluent
hypergeometric function.
(The existence and continuity of $f\ex$ follows also from \refT{P1} below.)
The series expansion of $f\ex$ readily yields asymptotics of the left tail of
the distribution,
\ie, of $f\ex(x)$ and $\P(\cBex\le x)$ as $x\to0$, see Louchard \cite{Lou:ex} 
and Flajolet and Louchard \cite{FL:Airy}
(with  typos
corrected in \cite{SJ201}). 

The main purpose of this paper is to give corresponding asymptotics
for the right tail of the distribution of $\cBex$,
\ie, for the 
density function $f\ex(x)$ and the
tail probabilities $\P(\cBex>x)$ for large $x$.
We have the following result.

\begin{theorem}\label{Tex}
For the Brownian excursion area,
as $\xtoo$,
  \begin{gather}\label{tex1}
  f\ex(x)\sim \frac{72\sqrt6}{\sqrt\pi} x^2 e^{-6x^2}
\intertext{and}
  \P(\cB\ex>x)\sim \frac{6\sqrt6}{\sqrt\pi} x e^{-6x^2}.
  \end{gather}
More precisely, there exist asymptotic expansions in powers of
$x\qww$, to arbitrary order $N$, as $\xtoo$,
  \begin{gather*}
  f\ex(x)= \frac{72\sqrt6}{\sqrt\pi} x^2 e^{-6x^2}
\lrpar{
1-\frac19\,{x}^{-2}-{\frac {5}{1296}}\,{x}^{-4}-{\frac {25}{46656}}\,{x}^{
-6}+\dots+O \left( {x}^{-2N} \right)
},
\\
  \P(\cB\ex>x)
= \frac{6\sqrt6}{\sqrt\pi} x e^{-6x^2}
\lrpar{
1-\frac1{36}\,{x}^{-2}-{\frac {1}{648}}\,{x}^{-4}-{\frac {7}{46656}}\,{x}^{-
6}+
\dots+
O \left( {x}^{-2N} \right)
}.
  \end{gather*}
\end{theorem}

Unlike the left tail, it seems difficult to obtain such results 
from \Takacs's formula for $f\ex$, and we will instead use a method by
Tolmatz
\cite{Tol:br,Tol:bm,Tol:brp} that he used to obtain corresponding
asymptotics for 
three other Brownian areas, \viz, 
the integral 
$\cB\br\=\intoi |B\br(t)|\dd t$
of the absolute value of a Brownian bridge $B\br(t)$,
the integral 
$\cB\bm\=\intoi |B(t)|\dd t$
of the absolute value of a Brownian motion $B(t)$ over $\oi$,
and the integral 
$\cB\brp\=\intoi B\br(t)_+\dd t$
of the positive part of a Brownian bridge.

The (much weaker) fact that 
$-\ln\P(\cB\ex>x)\sim{-6x^2}$, \ie, that
\begin{equation}
  \label{sofie}
\P(\cB\ex>x)=\exp\bigpar{-6x^2+o(x^2)},
\end{equation}
was shown by \citet{CsorgoShiYor}
as a consequence of the asymptotics of the moments $\E\cBex^n$ found
by \Takacs{} \cite{Takacs:Bernoulli}, see 
\refS{Smoments}. It seems difficult to obtain more precise tail
asymptotics from the moment asymptotics.
It is, however, easy to go in the opposite direction and obtain moment
asymptotics from the tail asymptotics above, as was done by Tolmatz
\cite{Tol:bm,Tol:brp}  for $\cB\bm$, $\cB\br$ and $\cB\brp$; see again
\refS{Smoments}. In particular, this made it possible to guess the
asymptotic formula \eqref{tex1} before we could prove it, by
matching the resulting moment asymptotics with the known 
result by \Takacs{} \cite{Takacs:Bernoulli}.

An alternative way to obtain \eqref{sofie} is by large deviation
theory, which easily gives \eqref{sofie} and explains the constant 6
as the result of an optimization problem, 
see Fill and Janson \cite{SJ197}. This method applies to the other
Brownian areas in this paper too, 
and explains the different constants in the exponents below,
but, again, it seems difficult to obtain more
precise results by this approach.

Besides the Brownian excursion area and the three areas studied by
Tolmatz, his method applies also to three further Brownian areas:
the integrals
$\cB\me\=\intoi |B\me(t)|\dd t$,
$\cB\dm\=\intoi |B\dm(t)|\dd t$
and 
$\cB\bmp\=\intoi B(t)_+\dd t$
of a Brownian meander $B\me(t)$,
a Brownian double meander $B\dm(t)$, and
the positive part of a Brownian motion
over $\oi$. 
We define here the Brownian double meander by
$B\dm(t)\=B(t)-\min_{0\le u\le 1}B(u)$;
this is a non-negative continuous stochastic process on $\oi$ that
\as{} is 0 at a unique point $\tau\in\oi$, and it can be regarded as
two Brownian meanders on the intervals 
$[0,\tau]$ and $[\tau,1]$ joined back to back (with the first one
reversed),
see \citet{MC:Airy} and \citet{SJ201}; the other processes considered
here are well-known, 
see for example \citet{RY}.

We find it illuminating to study all
seven Brownian areas together, and we will therefore formulate our
proof in a general form that applies to all seven areas.
As a result we obtain also the following results, where we for
completeness
repeat Tolmatz's results.
(We also extend them, since Tolmatz \cite{Tol:br,Tol:bm,Tol:brp}
gives only the leading terms, but
he points out that higher order terms can be obtained in the same way.)
We give the four first terms in the asymptotic expansions; they can
(in principle, at least)
be continued to any desired number of terms by the methods
in \refS{Shigher}; only even powers $x^{-2k}$ appear in the expansions.

\begin{theorem}[Tolmatz \cite{Tol:br}]
\label{Tbr}
For the Brownian bridge area,
as $\xtoo$,
  \begin{align*}
  f\br(x)&= \frac{2\sqrt6}{\sqrt\pi}\, e^{-6x^2}
\lrpar{
1+\frac1{18}\,{x}^{-2}+{\frac {1}{432}}\,{x}^{-4}
+{\frac {25}{46656}}\,{x}^{-6}
+O \left( {x}^{-8} \right)
},
\\
  \P(\cB\br>x)
&= \frac{1}{\sqrt{6\pi}}\, x\qw e^{-6x^2}
\lrpar{
1-\frac1{36}\,{x}^{-2}+{\frac {1}{108}}\,{x}^{-4}
-{\frac {155}{46656}}\,{x}^{-6}
+O \left( {x}^{-8} \right)
}.
  \end{align*}
\end{theorem}

\begin{theorem}[Tolmatz \cite{Tol:bm}]
\label{Tbm}
For the Brownian motion area,
as $\xtoo$,
  \begin{align*}
  f\bm(x)&= \frac{\sqrt6}{\sqrt\pi}\, e^{-3x^2/2}
\lrpar{
1+\frac 1{18}\,{x}^{-2}-{\frac {1}{162}}\,{x}^{-4}
+{\frac {49}{5832}}\,{x}^{-6}
+O \left( {x}^{-8} \right) 
},
\\
  \P(\cB\bm>x)&= \frac{\sqrt2}{\sqrt{3\pi}}\, x\qw e^{-3x^2/2}
\lrpar{
1-{\frac {5}{18}}\,{x}^{-2}+{\frac {22}{81}}\,{x}^{-4}
-{\frac {2591}{5832}}\,{x}^{-6}
+O \left( {x}^{-8} \right) 
}.
  \end{align*}
\end{theorem}

\begin{theorem}\label{Tme}
For the Brownian meander area,
as $\xtoo$,
  \begin{align*}
  f\me(x)&= 3\sqrt3\, x e^{-3x^2/2}
\lrpar{
1-\frac{1}{18}\,{x}^{-2}-{\frac {1}{162}}\,{x}^{-4}
+{\frac {5}{5832}}\,{x}^{-6}
+O \left( {x}^{-8} \right) 
},
\\
  \P(\cB\me>x)&= \sqrt{3}\,e^{-3x^2/2}
\lrpar{
1-\frac{1}{18}\,{x}^{-2}+{\frac {5}{162}}\,{x}^{-4}
-{\frac {235}{5832}}\,{x}^{-6}
+O \left( {x}^{-8} \right) 
}.
  \end{align*}
\end{theorem}

\begin{theorem}\label{Tdm}
For the Brownian double meander area,
as $\xtoo$,
  \begin{align*}
  f\dm(x)&= 
\frac{2\sqrt6}{\sqrt\pi}\, e^{-3x^2/2}
\lrpar{
1+\frac 1{6}\,{x}^{-2}+{\frac {1}{18}}\,{x}^{-4}
+{\frac {29}{648}}\,{x}^{-6}
+O \left( {x}^{-8} \right) 
},
\\
  \P(\cB\dm>x)&= 
\frac{2\sqrt2}{\sqrt{3\pi}}\, x\qw e^{-3x^2/2}
\lrpar{
1-{\frac {1}{6}}\,{x}^{-2}+{\frac {2}{9}}\,{x}^{-4}
-{\frac {211}{648}}\,{x}^{-6}
+O \left( {x}^{-8} \right) 
}.
  \end{align*}
\end{theorem}

\begin{theorem}[Tolmatz \cite{Tol:brp}]
\label{Tbrp}
For the positive part of Brownian bridge area,
as $\xtoo$,
  \begin{align*}
  f\brp(x)&= \frac{\sqrt6}{\sqrt\pi}\, e^{-6x^2}
\lrpar{
1+\frac1{36}\,{x}^{-2}-{\frac {7}{5184}}\,{x}^{-4}
+{\frac {17}{46656}}\,{x}^{-6}
+O \left( {x}^{-8}
 \right) 
},
\\
  \P(\cB\brp>x)&= \frac{1}{2\sqrt{6\pi}}\, x\qw e^{-6x^2}
\lrpar{
1-\frac{1}{18}\,{x}^{-2}+{\frac {65}{5184}}\,{x}^{-4}
-{\frac {907}{186624}}\,{x}^{-6}
+O \left( {x}^{-8}
 \right) 
}.
  \end{align*}
\end{theorem}

\begin{theorem}\label{Tbmp}
For the positive part of Brownian motion area,
as $\xtoo$,
  \begin{align*}
  f\bmp(x)&= \frac{\sqrt3}{\sqrt{2\pi}} \,e^{-3x^2/2}
\lrpar{
1+\frac1{36}\,{x}^{-2}-{\frac {5}{648}}\,{x}^{-4}
+{\frac {109}{15552}}\,{x}^{-6}
+O \left( {x}^{-8} \right) 
},
\\
  \P(\cB\bmp>x)&= \frac{1}{\sqrt{6\pi}}\, x\qw e^{-3x^2/2}
\lrpar{
1-{\frac {11}{36}}\,{x}^{-2}+{\frac {193}{648}}\,{x}^{-4}
-{\frac {2537}{5184}}\,{x}^{-6}
+O \left( {x}^{-8} \right) 
}.
  \end{align*}
\end{theorem}

It is not surprising that the tails are roughly Gaussian, with a decay
like $e^{-c x^2}$ for some constants $c$.
Note that the constant in the exponent is $6$ for the Brownian bridge
and excursion, which are tied to 0 at both endpoints, and only
$3/2$ for the Brownian motion, meander and double meander, which are
tied to 0 only at one point. 
It is intuitively clear that the probability of a very
large value is smaller in the former cases. There are also differences
in factors of $x$ between $\cB\br$ and $\cB\ex$, and between $\cB\bm$ and
$\cB\me$, where the process conditioned to be positive has somewhat
higher probabilities of large areas. These differences are in the
expected direction, but we see no intuitive reason for the powers in
the theorems. We have even less explanations for the constant factors
in the estimates.


\begin{remark}\label{R+}
  If we define
$\cB\brm\=\intoi B\br(t)_-\dd t$, we have
$\cB\br=\cB\brp+\cB\brm$;
further, $\cB\brm\eqd\cB\brp$ by symmetry.
Hence, for any $x$,
\begin{equation*}
  \begin{split}
  \P(\cB\br>x)
&\ge
  \P(\cB\brp>x \text{ or }\cB\brm>x)
\\&
=
\P(\cB\brp>x)+\P(\cB\brm>x)- \P(\cB\brp>x \text{ and }\cB\brm>x)
\\&
\ge
2\P(\cB\brp>x)-2\P(\cB\br>2x).	
  \end{split}
\end{equation*}
By Theorems \refand{Tbr}{Tbrp}, the ratio between the two sides is
$1+\tfrac1{36}x^{-2}+O(x^{-4})$;
hence, these inequalities are tight for large $x$. This shows, in a
very precise way, the intuitive fact that the most probable way to
obtain a large value of $\cB\br$ is with one of $\cB\brp$ and
$\cB\brm$ large and the other close to 0.

The same is true for $\cB\bm$ and $\cB\bmpm$ 
by Theorems \refand{Tbr}{Tbrp}. It is interesting to note that
for both $\cB\br$ and $\cB\bm$,
the ratio 
$\P(\cB>x)/2\P(\cB_+)=1+\tfrac1{36}x^{-2}+O(x^{-4})$, 
with the first two terms equal for the two cases
(the third terms differ).
\end{remark}

Tolmatz's method is based on inverting a double Laplace transform;
this double Laplace transform has simple explicit forms (involving the
Airy function) for all seven Brownian areas, see the survey
\cite{SJ201} and the references given there. The inversion is far from
trivial; a straightforward inversion leads to a double integral that
is not even absolutely convergent, and not easy to estimate.
Tolmatz found a clever change of contour that together with
properties of the Airy function leading to
near cancellations 
makes it possible to rewrite the integral as a double integral of a
rapidly decreasing function, for which the saddle point method can be
applied. 
(\citet{KearneyMM} have recently used a similar change of contour 
together with similar near cancellations
to invert a (single) Laplace transform for another
type of Brownian area.) 
We follow Tolmatz's approach, and state his inversion using a 
change of contour in a rather general form in \refS{SP1}; the proof is
given in \refS{SproofP1}.
This inversion formula is then applied to the seven Brownian areas in
Sections \ref{SPsi}--\ref{Shigher}.
Moment asymptotics are derived in \refS{Smoments}.

A completely different proof for the asymptotics of $\P(\cB\br>x)$ and
$\P(\cB\bm>x)$ in Theorems \refand{Tbr}{Tbm} has been given by Fatalov
\cite{Fatalov} using Laplace's method in Banach spaces. 
This method seems to be an interesting and flexible alternative way to
obtain at least first order asymptotics in many situations, and it
would be interesting to extend it to cover all cases treated here.

We use $C_1,C_2,\dots$ 
and $c_1,c_2,\dots$ to denote various positive
constants; explicit values could be given but are unimportant. We also
write, for example, $C_1(M)$ to denote dependency on a parameter (but
not on anything else).

\section{Asymptotics of density and distribution functions}\label{Stails}

The relation between the asymptotics for density functions and
distribution functions in Theorems \ref{Tex}--\ref{Tbmp} can be
obtained as follows.

Suppose that $X$ is a positive random variable with a density function
$f$ satisfying
\begin{equation}\label{fas}
  f(x)\sim ax^\ga e^{-bx^2}, \qquad \xtoo,
\end{equation}
for some numbers $a,b>0$, $\ga\in\bbR$.
It is easily seen, \eg{} by integration by parts, that \eqref{fas} implies
\begin{equation}\label{Fas}
  \P(X>x)\sim \frac{a}{2b}x^{\ga-1} e^{-bx^2}, \qquad \xtoo.
\end{equation}
Obviously, there is no implication in the opposite direction; $X$ may
even satisfy \eqref{Fas} without having a density at all. On the other
hand, if it is known that \eqref{fas} holds with some unknown
constants $a$, $b$, $\ga$, then the constants can be found from the
asysmptotics of $\P(X>x)$ by \eqref{Fas}.

The argument extends to asymptotic expansions with higher order
terms. If, as for the Brownian areas studied in this paper, there is
an asymptotic expansion
\begin{equation}\label{fasx}
  f(x)= x^\ga e^{-bx^2}
\lrpar{a_0+a_2x^{-2}+a_4x^{-4}+\dots+O(x^{-2N})}, 
\qquad \xtoo, 
\end{equation}
then repeated integrations by parts yield a corresponding expansion
\begin{equation}\label{Fasx}
  \P(X>x)= x^{\ga-1} e^{-bx^2}
\lrpar{a'_0+a'_2x^{-2}+a'_4x^{-4}+\dots+O(x^{-2N})}, 
\qquad \xtoo,
\end{equation}
where $a'_0=a_0/(2b)$, $a'_2=a_0(\ga-1)/(2b)^2+a_2/(2b)$, \dots;
in general, the expansion \eqref{fasx} is recovered by formal
differentiation of \eqref{Fasx}, which gives a simple method to find
the coefficients in \eqref{Fasx}.

\section{A double Laplace inversion}\label{SP1}

We state the main step in (our version of)
Tolmatz' method as the following inversion
formula, which is based on and generalizes formulas in Tolmatz
\cite{Tol:br,Tol:bm,Tol:brp}. 

Fractional powers of complex numbers below are interpreted as the
principal values, defined in $\bbC\setminus(-\infty,0]$.

\begin{theorem}\label{P1}
Let $X$ be a positive random variable 
and let $\psi(s)\=\E e^{-sX}$ be its
  Laplace transform.
Suppose that  $0<\nu<3/2$ and that
\begin{equation}
  \label{a1}
\frac1{\Gamma(\nu)}\intoo e^{-xs}\psi(s\qqc)s^{\nu-1}\dd s
=\Psi(x),
\qquad x>0,
\end{equation}
where $\Psi$ is an analytic function in 
the sector \set{z\in\bbC:|\arg z|<5\pi/6}
such that 
  \begin{align}
	\label{H0}
\Psi(z)&=o(|z|^{-\nu}), & z&\to0 \text{ with } |\arg z|< 5\pi/6,
\\
	\label{Hoo}
\Psi(z)&=O(1), & |z|&\to\infty \text{ with } |\arg z|< 5\pi/6.
  \end{align}
Let
\begin{equation}
  \label{Hx}
\Hx(z)\=
e^{2\pi\nu\ii/3}  \Psi\bigpar{e^{2\pi\ii/3}z}
-
e^{-2\pi\nu\ii/3}  \Psi\bigpar{e^{-2\pi\ii/3}z}.
\end{equation}
Finally, assume that 
  \begin{align}
	\label{Hxoox}
\Hx(z)&=O(|z|^{-6}), & |z|&\to\infty \text{ with } |\arg z|<\pi/6.
  \end{align}
Then $X$ is absolutely continuous with a continuous density function
$f$ given by, for $\glx>0$ and every $\xi>0$,
{
\begin{multline}\label{ff}
f(\glx)
=\frac{3\Gamma(\nu)}{8\pi^2\ii}
\xi^{5/2-\nu}\glx^{2\nu/3-5/3}
\\
\shoveright{\cdot
\int_{\gt=-\pi/2}^{\pi/2}\int_{r=0}^\infty
\exp\Bigpar{\xi\glx\qqaw
\sec\gt e^{\ii\gt}-e^{\ii\gt}(\xi\sec\gt)\qqc r\qqcw}}
\\ \cdot
e^{(1-2\nu/3)\ii\gt}
\xpar{\sec\gt}^{7/2-\nu} 
r^{\nu-5/2} \Hx\xpar{re^{\ii\gt/3}}
\dd r \dd\gt.
\end{multline}}
\end{theorem}

Note that $\Hx$ is analytic in the
sector $|\arg z|<\frac\pi6$, with, by \eqref{H0} and \eqref{Hoo}, 
  \begin{align}
	\label{Hx0}
\Hx(z)&=o(|z|^{-\nu}), && z\to0 \text{ with } |\arg z|<\frac\pi6,
\\
	\label{Hxoo}
\Hx(z)&=O(1), && |z|\to\infty \text{ with } |\arg z|<\frac\pi6.
  \end{align}
However, we need, as assumed in \eqref{Hxoox}, 
a more rapid decay as $|z|\to\infty$ than this.
\begin{remark}
  \label{Rdelta}
In all our applications, $\Psi$ is, in fact, analytic 
in the slit plane $\bbC\setminus(-\infty,0]$, and \eqref{H0} and
  \eqref{Hoo} hold in any sector ${|\arg z|\le\pi-\gd}$; thus $\Hx$ is
  analytic in ${|\arg z|<\pi/3}$, and \eqref{Hx0} and
  \eqref{Hxoo} hold for
$|\arg z|\le\pi/3-\gd$.
\end{remark}

\begin{remark} \label{RTol}
To obtain Tolmatz' version of the formulas, for example \cite[(30)]{Tol:br}
(correcting a typo there), 
take 
$\nu=1/2$ and $\Hx$ as in \eqref{Hxbr0} below, and
make the substitutions $x=\gl$, $\xi=a\gl\qqa$ and 
$r=a\rho\qqaw\sec\gt$.   
\end{remark}

We prove \refT{P1} in \refS{SproofP1}, 
but show first how it applies
to the Brownian areas.
 
\section{The function $\Hx$ for Brownian areas}\label{SPsi}

For the Brownian bridge area $\cB\br$ we have $\nu=1/2$ and,
see e.g. \cite[(126)]{SJ201},
\begin{equation}\label{Hbr0}
  \Psi(z)=-2^{1/6}\frac{\Ai(2\qqq z)}{\Ai'(2\qqq z)},
\end{equation}
which by the formula \cite[10.4.9]{AS} 
\begin{equation}\label{A2pi/3}
  \Ai(ze^{\pm2\pi\ii/3})=\tfrac12e^{\pm\pi\ii/3}\bigpar{\Ai(z)\mp\ii\Bi(z)}
\end{equation}
and its consequence
\begin{equation}\label{A'2pi/3}
  \Ai'(ze^{\pm2\pi\ii/3})=\tfrac12e^{\mp\pi\ii/3}\bigpar{\Ai'(z)\mp\ii\Bi'(z)}
\end{equation}
together with the Wronskian \cite[10.4.10]{AS} 
\begin{equation}
  \label{Wronskian}
\Ai(z)\Bi'(z)-\Ai'(z)\Bi(z)=\pi\qw
\end{equation}
by a simple calculation
leads to, 
as shown by Tolmatz \cite[Lemma 2.1]{Tol:br}, see \eqref{Hxbr} below,
\begin{equation}\label{Hxbr0}
  \Hx(z)=\frac{2^{7/6}\pi\qw\ii}{\Ai'(2\qqq z)^2+\Bi'(2\qqq z)^2}.
\end{equation}
It seems simpler to instead consider $\sqrtt\cBbr$.
Note that, by the simple change of variables $s\mapsto2\qqq s$ and
$x\mapsto2\qqqw x$ in \eqref{a1}, if
\eqref{a1} holds for some random variable $X$ and a function $\Psi$, it
holds for $\sqrtt X$ and $2^{-\nu/3}\Psi(2\qqqw z)$.
We use the notations $\Psi\br$ and $\Hx\br$ for the case $X=\sqrtt\cBbr$
and obtain from \eqref{Hbr0} the simpler
\begin{equation}\label{Hbr}
  \Psi\br(z)=-\frac{\Ai(z)}{\Ai'(z)}
\end{equation}
and thus, by \eqref{A2pi/3}, \eqref{A'2pi/3} and \eqref{Wronskian},
\begin{equation}\label{Hxbr}
  \begin{split}
  \Hx\br(z)
&=
\sum_{\pm}\pm
e^{\pm\pi\ii/3}\Psi\bigpar{e^{\pm2\pi\ii/3}z}
=
\sum_{\pm}
\mp e^{\pm3\pi\ii/3}
 \frac{\Ai(z)\mp\ii\Bi(z)}{\Ai'(z)\mp\ii\Bi'(z)}
\\&
=
\sum_{\pm}\pm
 \frac{\bigpar{\Ai(z)\mp\ii\Bi(z)}\bigpar{\Ai'(z)\pm\ii\Bi'(z)}}
{\Ai'(z)^2+\Bi'(z)^2}
\\&
=\frac{2\pi\qw\ii}{\Ai'(z)^2+\Bi'(z)^2}.	
  \end{split}
\end{equation}

For the Brownian excursion area $\cB\ex$ we have $\nu=1/2$ and by
Louchard \cite{Lou:ex}, see also \cite[(80)]{SJ201},
\begin{equation}\label{Hex0}
\Psi(z)
=-2^{5/6}\frac{\dd}{\dd z}\lrpar{\frac{\Ai'(2\qqq z)}{\Ai(2\qqq z)}}
=2\qq\lrpar{2\qqq\frac{\Ai'(2\qqq z)}{\Ai(2\qqq z)}}^2-2\qqc z,
\end{equation}
Again, it seems simpler to instead consider $\sqrtt\cBex$, for which we 
use the notation $\Psi\ex$ and $\Hx\ex$. We
have, see
Louchard \cite{Lou:ex} and \cite[(81)]{SJ201}, or by \eqref{Hex0}
and the general relation above,
\begin{equation}\label{Hex}
\Psi\ex(z)
=-2\frac{\dd}{\dd z}\lrpar{\frac{\Ai'(z)}{\Ai(z)}}
=2\lrpar{\frac{\Ai'(z)}{\Ai(z)}}^2-2z, 
\end{equation}
and thus by \eqref{A2pi/3}, \eqref{A'2pi/3} and \eqref{Wronskian}
\begin{equation}\label{Hxex}
  \begin{split}
\Hx\ex(z)
&=
\sum_{\pm}\pm
e^{\pm\pi\ii/3}\Psi\ex\bigpar{e^{\pm2\pi\ii/3}z}
\\&
=
\sum_{\pm}\pm \biggpar{
e^{\pm\pi\ii/3}2
 \lrpar{e^{\mp2\pi\ii/3}\frac{\Ai'(z)\mp\ii\Bi'(z)}{\Ai(z)\mp\ii\Bi(z)}}^2
-2e^{\pm2\pi\ii/3}z}
\\&
=
2\sum_{\pm}\pm
e^{\mp3\pi\ii/3}
 \frac{\bigpar{\Ai'(z)\mp\ii\Bi'(z)}^2\bigpar{\Ai(z)\pm\ii\Bi(z)}^2}
   {\bigpar{\Ai(z)^2+\Bi(z)^2}^2}
+0
\\&
=
2\sum_{\pm}\mp
 \frac{\bigpar{\Ai'(z)\Ai(z)+\Bi'(z)\Bi(z)\mp\ii\pi\qw}^2}
   {\bigpar{\Ai(z)^2+\Bi(z)^2}^2}
\\&
=  \frac{8\pi\qw\ii\bigpar{\Ai(z)\Ai'(z)+\Bi(z)\Bi'(z)}}
   {\bigpar{\Ai(z)^2+\Bi(z)^2}^2}.
  \end{split}
\end{equation}

The Brownian motion area $\cBbm$
is another case treated by
Tolmatz \cite{Tol:bm}.
Note that in this case $\nu=1$.
For $\sqrtt\cB\bm$, we have
by \Takacs{} \cite{Takacs:BM}, see also 
Kac \cite{Kac46},
\PW, and \cite[Section 20 and
  Appendix C.1]{SJ201}, 
$\nu=1$ and
\begin{equation}\label{Hbm}
\Psi\bm(z)
=
-\frac{\AI(z)}{\Ai'(z)},
\end{equation}
where we use the notation, see \cite[Appendix A]{SJ201},
\begin{equation}
\AI(z)\=\int_z^{+\infty} \Ai(t)\dd t
=\frac13-\int_0^z \Ai(t)\dd t.
\end{equation}
If we further define
\begin{equation}
  \label{BI}
\BI(z)\=\int_0^z \Bi(t)\dd t,
\end{equation}
we have by \eqref{A2pi/3}
\begin{equation}  \label{AI2pi/3}
  \begin{split}
\AI(ze^{\pm2\pi\ii/3})
&=
\frac13-\int_0^{ze^{\pm2\pi\ii/3}} \Ai(t)\dd t
\\&
=
\frac13-e^{\pm2\pi\ii/3}\int_0^{z} \Ai(te^{\pm2\pi\ii/3})\dd t
\\&
=
\frac13-\frac12e^{\pm3\pi\ii/3}\int_0^{z} \bigpar{\Ai(t)\mp\ii\Bi(t)}\dd t
\\&
=
\half-\half\AI(z)\mp\half\ii\BI(z).
  \end{split}
\end{equation}
Consequently, using \eqref{Hbm} and \eqref{A'2pi/3},
\begin{equation}\label{Hxbm}
  \begin{split}
\Hx\bm(z)
&=
\sum_{\pm}\pm
e^{\pm2\pi\ii/3}
\Psi\bm\bigpar{e^{\pm2\pi\ii/3}z}
\\&
=
\sum_{\pm}\mp e^{\pm3\pi\ii/3}
\frac{1-\AI(z)\mp\ii\BI(z)}{\Ai'(z)\mp\ii\Bi'(z)}
\\&
=
\sum_{\pm}\pm 
\frac{\bigpar{1-\AI(z)\mp\ii\BI(z)}\bigpar{\Ai'(z)\pm\ii\Bi'(z)}}
{\Ai'(z)^2+\Bi'(z)^2}
\\&
=
2\ii
\frac{\Bi'(z)-\AI(z)\Bi'(z)-\Ai'(z)\BI(z)}
{\Ai'(z)^2+\Bi'(z)^2}
.
  \end{split}
\end{equation}

For the Brownian meander, or more precisely $\sqrtt\cB\me$, by
\Takacs{} \cite{Takacs:meander}, see also \cite[Section 22 and
  Appendix C.3]{SJ201},
\eqref{a1} holds with $\nu=1/2$ and
\begin{equation}\label{Hme}
\Psi\me(z)=\frac{\AI(z)}{\Ai(z)}.
\end{equation}
Consequently, using \eqref{AI2pi/3} and \eqref{A2pi/3},
\begin{equation}\label{Hxme}
  \begin{split}
\Hx\me(z)
&=
\sum_{\pm}\pm
e^{\pm\pi\ii/3}\Psi\me\bigpar{e^{\pm2\pi\ii/3}z}
\\&
=
\sum_{\pm}\pm 
\frac{1-\AI(z)\mp\ii\BI(z)}{\Ai(z)\mp\ii\Bi(z)}
\\&
=
\sum_{\pm}\pm 
\frac{\bigpar{1-\AI(z)\mp\ii\BI(z)}\bigpar{\Ai(z)\pm\ii\Bi(z)}}
{\Ai(z)^2+\Bi(z)^2}
\\&
=
2\ii
\frac{\Bi(z)-\AI(z)\Bi(z)-\Ai(z)\BI(z)}
{\Ai(z)^2+\Bi(z)^2}
.
  \end{split}
\end{equation}

For the Brownian double meander, or more precisely $\sqrtt\cB\dm$, by
\citet{MC:Airy}, see also \cite[Section 23]{SJ201},
\eqref{a1} holds with $\nu=1$ and
\begin{equation}\label{Hdm}
\Psi\dm(z)=\lrpar{\frac{\AI(z)}{\Ai(z)}}^2.
\end{equation}
Consequently, using \eqref{AI2pi/3} and \eqref{A2pi/3},
\begin{equation}\label{Hxdm}
  \begin{split}
\Hx\dm(z)
&=
\sum_{\pm}\pm
e^{\pm2\pi\ii/3}\Psi\dm\bigpar{e^{\pm2\pi\ii/3}z}
\\&
=
\sum_{\pm}\pm 
\lrpar{
\frac{1-\AI(z)\mp\ii\BI(z)}{\Ai(z)\mp\ii\Bi(z)}
}^2
\\&
=
\sum_{\pm}\pm 
\frac{\bigpar{\bigpar{1-\AI(z)\mp\ii\BI(z)}\bigpar{\Ai(z)\pm\ii\Bi(z)}}^2}
{\lrpar{\Ai(z)^2+\Bi(z)^2}^2}
\\&\hskip-2em
=
4\ii
\frac{\bigpar{\xpar{1-\AI(z)}\Ai(z)+\BI(z)\Bi(z)}\bigpar{(1-\AI(z))\Bi(z)-\BI(z)\Ai(z)}}
{\lrpar{\Ai(z)^2+\Bi(z)^2}^2}
.
  \end{split}
\end{equation}

The positive part of a Brownian bridge is another case treated by 
Tolmatz \cite{Tol:brp}.
For $\sqrtt\cB\brp$, by \PW,
see also Tolmatz \cite{Tol:brp} and \cite[Section 22 and Appendix C.2]{SJ201},
\eqref{a1} holds with $\nu=1/2$ and
\begin{equation}\label{Hbrp}
\Psi\brp(z)
=
2
\frac{\Ai(\glz)}
{\glz\qq\Ai(\glz)-\Ai'(\glz)}.
\end{equation}
Consequently,
by \eqref{A2pi/3}, \eqref{A'2pi/3} and \eqref{Wronskian},
\begin{equation}\label{Hxbrp}
  \begin{split}
\Hxbrp(z)
&=
\sum_{\pm}\pm
e^{\pm\pi\ii/3}
\Psi\brp\bigpar{e^{\pm2\pi\ii/3}z}
\\&
=
\sum_{\pm}\pm 2e^{\pm\pi\ii/3}
\frac{\Ai\bigpar{e^{\pm2\pi\ii/3}z}}
{e^{\pm\pi\ii/3}
z\qq\Ai\bigpar{e^{\pm2\pi\ii/3}z}-\Ai'\bigpar{e^{\pm2\pi\ii/3}z}}
\\&
=
2\sum_{\pm}\pm 
\frac
{\Ai(z)\mp\ii\Bi(z)}
{z\qq\bigpar{\Ai(z)\mp\ii\Bi(z)}
 -e^{\mp3\pi\ii/3}\bigpar{\Ai'(z)\mp\ii\Bi'(z)}}
\\&
=
2\sum_{\pm}\pm 
\frac
{\Ai(z)\mp\ii\Bi(z)}
{\bigpar{z\qq\Ai(z)+\Ai'(z)}
\mp\ii\bigpar{z\qq\Bi(z)+\Bi'(z)}}
\\&
=
2\sum_{\pm}\pm 
\frac
{\bigpar{\Ai(z)\mp\ii\Bi(z)}\bigpar{z\qq\Ai(z)+\Ai'(z)
\pm\ii\bigpar{z\qq\Bi(z)+\Bi'(z)}}}
{\bigpar{z\qq\Ai(z)+\Ai'(z)}^2
+\bigpar{z\qq\Bi(z)+\Bi'(z)}^2}
\\&
=
\frac
{4\ii\pi\qw}
{\bigpar{z\qq\Ai(z)+\Ai'(z)}^2
+\bigpar{z\qq\Bi(z)+\Bi'(z)}^2}.
  \end{split}
\raisetag{\baselineskip}
\end{equation}

For the positive part of a Brownian motion, or more precisely
$\sqrtt\cB\bmp$, by \PW,
see also \cite[Section 23 and Appendix C.1]{SJ201},
\eqref{a1} holds with $\nu=1$ and
\begin{equation}\label{Hbmp}
\Psi\bmp(z)
=
\frac{\glz\qqw\Ai(\glz)+\AI(\glz)}
{\glz\qq\Ai(\glz)-\Ai'(\glz)}
.
\end{equation}
Note that this $\Psi$ is singular at 0, but still satisfies \eqref{H0}.
By \eqref{A2pi/3}, \eqref{A'2pi/3}, \eqref{AI2pi/3} and \eqref{Wronskian},
\begin{equation}\label{Hxbmp}
  \begin{split}
\Hxbmp(z)
&=
\sum_{\pm}\pm
e^{\pm2\pi\ii/3}
\Psi\bmp\bigpar{e^{\pm2\pi\ii/3}z}
\\&
=
\sum_{\pm}\pm e^{\pm2\pi\ii/3}
\frac{e^{\mp\pi\ii/3}
z\qqw\Ai\bigpar{e^{\pm2\pi\ii/3}z}+\AI\bigpar{e^{\pm2\pi\ii/3}z}}
{e^{\pm\pi\ii/3}
z\qq\Ai\bigpar{e^{\pm2\pi\ii/3}z}-\Ai'\bigpar{e^{\pm2\pi\ii/3}z}}
\\&
=
\sum_{\pm}\pm 
\frac
{z\qqw\bigpar{\Ai(z)\mp\ii\Bi(z)}+
{1-\AI(z)\mp\ii\BI(z)}}
{z\qq\bigpar{\Ai(z)\mp\ii\Bi(z)}
 -e^{\mp3\pi\ii/3}\bigpar{\Ai'(z)\mp\ii\Bi'(z)}}
\\&
=
\sum_{\pm}\pm 
\frac
{z\qqw\Ai(z)+1-\AI(z)\mp\ii\bigpar{z\qqw\Bi(z)+\BI(z)}}
{\bigpar{z\qq\Ai(z)+\Ai'(z)}
\mp\ii\bigpar{z\qq\Bi(z)+\Bi'(z)}}
\\&
=
2\ii
\frac
{
\bigpar{z\qqw\Ai(z)+1-\AI(z)}\bigpar{z\qq\Bi(z)+\Bi'(z)}
}
{\bigpar{z\qq\Ai(z)+\Ai'(z)}^2
+\bigpar{z\qq\Bi(z)+\Bi'(z)}^2}
\\&
\hskip6em{}-
2\ii
\frac
{
\bigpar{z\qqw\Bi(z)+\BI(z)}\bigpar{z\qq\Ai(z)+\Ai'(z)}}
{\bigpar{z\qq\Ai(z)+\Ai'(z)}^2
+\bigpar{z\qq\Bi(z)+\Bi'(z)}^2}
\\&
\hskip-4em
=
2\ii
\frac
{
\bigpar{1-\AI(z)}\bigpar{z\qq\Bi(z)+\Bi'(z)}
-\BI(z)\bigpar{z\qq\Ai(z)+\Ai'(z)}+z\qqw\pi\qw}
{\bigpar{z\qq\Ai(z)+\Ai'(z)}^2
+\bigpar{z\qq\Bi(z)+\Bi'(z)}^2}
  \end{split}
\end{equation}

Note that the functions $\Psi\br$, $\Psi\ex$, $\Psi\bm$,
$\Psi\me$ and $\Psi\dm$
given above
in \eqref{Hbr}, \eqref{Hex}, \eqref{Hbm}, \eqref{Hme}, \eqref{Hdm}
are meromorphic, with poles only on the negative real axis, because
the only zeros of $\Ai$ and $\Ai'$ are on the negative real axis 
\cite[p.\ 450]{AS}. The functions $\Psi\brp$ and $\Psi\bmp$ 
in \eqref{Hbrp} and \eqref{Hbmp}
are
analytic in the slit plane
$\bbC\setminus(-\infty,0]$, since Tolmatz \cite{Tol:brp} showed
that $z\qq\Ai(z)-\Ai'(z)$ has no zeros in the slit plane;
see \refA{Azeros} for an alternative proof. 
In particular, all seven functions are analytic in the slit plane.
Furthermore, all except $\Psi\bmp$ have finite limits as
$z\to0$, and in particular they are $O(1)$ as $z\to0$ so \eqref{H0}
holds. By \eqref{Hbmp}, we have $\Psi\bmp(z)\sim z\qqw\Ai(0)/\Ai'(0)$
and thus $\Psi\bmp=O(|z|\qqw)$ as $z\to0$; since in this case $\nu=1$,
\eqref{H0} holds for $\Psi\bmp$ too.

Next we consider asymtotics as $|z|\to\infty$.
The Airy functions have well-known asymptotics, see
\cite[10.4.59, 10.4.61, 10.4.63, 10.4.66, 10.4.82, 10.4.84]{AS}. The
leading terms are, 
as $|z|\to\infty$ and uniformly in the indicated sectors for any $\gd>0$,
\begin{align}
\Ai(z)&\sim  \label{Aioo}
\frac{\pi^{-1/2}}{2}z^{-1/4}e^{-\zetax},
&&|\arg(z)|\le\pi-\gd,\\
\Ai'(z)&\sim \label{Aiioo}
-\frac{\pi^{-1/2}}{2}z^{1/4}e^{-\zetax},
&&|\arg(z)|\le\pi-\gd,\\
\AI(z)&\sim  \label{AIoo}
\frac{\pi^{-1/2}}{2}z^{-3/4}e^{-\zetax},
&&|\arg(z)|\le\pi-\gd,\\
\Bi(z)&\sim \label{Bioo}
\pi^{-1/2}z^{-1/4}e^{\zetax}
,
&&|\arg(z)|\le\pi/3-\gd,\\
\Bi'(z)&\sim  \label{Biioo}
\pi^{-1/2}z^{1/4}e^{\zetax}
,
&&|\arg(z)|\le\pi/3-\gd,\\
\BI(z)&\sim   \label{BIoo}
\pi^{-1/2}z^{-3/4}e^{\zetax},
&&|\arg(z)|\le\pi/3-\gd.
\end{align}

It follows by using \eqref{Aioo}, \eqref{Aiioo} and \eqref{AIoo} in
\eqref{Hbr}, \eqref{Hex}, \eqref{Hme}, \eqref{Hbm}, \eqref{Hbrp},
\eqref{Hbmp} that in all seven cases 
\eqref{Hoo} holds; more precisely, 
$\Psi(z)\sim z^{-\nu}$ as $|z|\to\infty$ with $|\arg z|<\pi-\gd$.
(For real $z>0$, this is always true, as follows from \eqref{a1} by
the change of variables $s=t/x$ and monotone (or dominated) convergence.)

Turning to $\Hx$, we observe first that, by \eqref{Hx}, in all seven
cases, $\Hx(z)$ is analytic in $|\arg z|<1/3$. 
Next, 
\eqref{Aioo}--\eqref{BIoo} show that, as $|z|\to\infty$ in a sector 
$|\arg(z)|\le\pi/3-\gd$, $\Ai,\Ai',\AI$ decrease superexponentially
while $\Bi,\Bi',\BI$ increase superexponentially.
Hence, we can ignore all terms involving $\Ai$.
More precisely,
\eqref{Hxbr}, \eqref{Hxex}, \eqref{Hxbm}, \eqref{Hxme}, \eqref{Hxdm}, 
\eqref{Hxbrp},
\eqref{Hxbmp} 
together with \eqref{Aioo}--\eqref{BIoo} 
yield the asymptotics, as $|z|\to\infty$
with (for example) $|\arg z|\le \pi/6$,
\begin{align}
\label{psixbr}
 \Hx\br(z)
&=
\frac{2\pi\qw\ii}{\Bi'(z)^2}\Bigpar{1+O\lrpar{e^{-8z\qqc/3}}},
\\ 
\label{psixex}
\Hx\ex(z)
&=
8\pi\qw\ii\,
  \frac{\Bi'(z)}
   {\Bi(z)^3}
\Bigpar{1+O\lrpar{e^{-8z\qqc/3}}},
\\
\label{psixbm}
\Hx\bm(z)
&=
\frac{2\ii}
{\Bi'(z)}
\Bigpar{1+O\lrpar{e^{-2z\qqc/3}}},
\\
\label{psixme}
\Hx\me(z)
&=
\frac{2\ii}
{\Bi(z)}
\Bigpar{1+O\lrpar{e^{-2z\qqc/3}}},
\\
\label{psixdm}
\Hx\dm(z)
&=4\ii\,
\frac{\BI(z)}
{\Bi(z)^2}
\Bigpar{1+O\lrpar{e^{-2z\qqc/3}}},
\\
\label{psixbrp}
\Hxbrp(z)
&=
\frac
{4\ii\pi\qw}
{\bigpar{z\qq\Bi(z)+\Bi'(z)}^2}
\Bigpar{1+O\lrpar{e^{-4z\qqc/3}}},
\\
\label{psixbmp}
\Hxbmp(z)
&=
\frac
{2\ii}
{z\qq\Bi(z)+\Bi'(z)}
\Bigpar{1+O\lrpar{e^{-\zetax}}}
.
\end{align}
In all seven cases, $\Hx$ decreases superexponentially in the sector; in
particular, \eqref{Hxoox} holds.
It is remarkable that in all seven cases, $\Psi(z)$ decreases slowly, as
$z\qqw$ or $z\qw$, but the linear combination $\Hx(z)$ decreases
extremely rapidly in a sector around the positive real axis; there are
thus almost complete cancellations between the values of $\Psi(z)$ at,
say, $\arg z=\pm2\pi\ii/3$. These cancellations are an important part
of the success of Tolmatz's method.

We have verified all the conditions of \refT{P1}. Hence, the theorem
shows that the variables have continuous density functions given by \eqref{ff}.

\section{The saddle point method}\label{Ssaddle}

We proceed to show how the tail asymptotics for the Brownian areas
follow from \refT{P1} and the formulae in \refS{SPsi} by
straightforward applications of the saddle point method. For
simplicity, we give first a derivation of the leading terms. In the
next section we show how the calculations can be refined to obtain
the asymptotic expansions in Theorems \ref{Tex}--\ref{Tbmp}.

We use $\Xi\in\set{\xbr,\,\xex,\, \xbm,\, \xme,\, \xdm,\, \xbrp,\, \xbmp}$ 
as a variable 
indicating the different Brownian areas we consider.
We begin by writing \eqref{psixbr}--\eqref{psixbmp},
using \eqref{Bioo} and \eqref{Biioo}, as 
\begin{equation}\label{psih}
  \Hx\xxi(z)=h\xxi(z)e^{-\gam\xxi z\qqc},
\end{equation}
where $\gam\br=\gam\ex=\gam\brp=4/3$ and $\gam\bm=\gam\me=\gam\dm=\gam\bmp=2/3$
(note that these cases differ by having two or one points tied to
0)
and, as $|z|\to\infty$ with $|\arg z|\le\pi/6$,
\begin{align}
\label{hbr}
h\br(z)
&\sim 2\ii z\qqw
\\ 
\label{hex}
h\ex(z)
&\sim
8\ii z
\\
\label{hbm}
h\bm(z)
&\sim
2\ii\pi\qq z\qqqqw
\\
\label{hme}
h\me(z)
&\sim
2\ii\pi\qq z\qqqq
\\
\label{hdm}
h\dm(z)
&\sim
4\ii\pi\qq z\qqqqw
\\
\label{hbrp}
h\brp(z)
&\sim
\ii z\qqw
\\
\label{hbmp}
h\bmp(z)
&\sim
\ii\pi\qq z\qqqqw
.
\end{align}
We write the \rhs{s} as $\ii\hh\br(z),\dots,\ii\hh\bmp(z)$,
and thus these formulae can be written
\begin{equation}\label{hhxi}
  h\xxi(z)\sim\ii\hh\xxi(z),
\end{equation}
where $\hh\br(z)=2z\qqw$, $\hh\ex(z)=8z$, and so on.

Consider, for simplicity, first the cases 
$\Xi\in\set{\xbr,\,\xex,\, \xme,\, \xbrp}$ 
where $\nu=1/2$.
We then rewrite \eqref{ff} as, using $\fx\xxi$ for the density of
$\sqrtt\cB\xxi$, 
\begin{equation}
\fx\xxi(\glx)
=
\xi^{2}\glx^{-4/3}
\int_{-\pi/2}^{\pi/2}\int_{0}^\infty
F_0(r,\gt)
e^{\phi_0(r,\gt;x,\xi)}
\dd r \dd\gt
\end{equation}
where, with $\gamxi=\gam\xxi$,
\begin{align}
  F_0(r,\gt)&\=
\frac{3\pi\qqcw}{8\ii}e^{2\ii\gt/3}(\sec\gt)^3r\qww h\xxi\xpar{re^{\ii\gt/3}},
\\
\phi_0(r,\gt;x,\xi)&\=
\xi\glx\qqaw\sec\gt e^{\ii\gt}-e^{\ii\gt}(\xi\sec\gt/r)\qqc
-\gamxi r\qqc e^{\ii\gt/2}.
\end{align}
Remember that $\xi$ is arbitrary;
we choose 
$\xi=\rho x^{8/3}$ 
for a positive constant  $\rho$ that will be chosen later.
Further, make the change of variables $r=x^{4/3} s\qqa$.
Thus,
\begin{equation}\label{fxia}
\fx\xxi(\glx)
=
\rho^{2}\glx^{8/3}
\int_{\gt=-\pi/2}^{\pi/2}\int_{s=0}^\infty
F_1(s,\gt;x)
e^{x^2\phi_1(s,\gt)}
\dd s \dd\gt
\end{equation}
where
\begin{align}
\label{F1a}
  F_1(s,\gt;x)&\=
\frac{1}{4\pi\qqc\ii}e^{2\ii\gt/3}(\sec\gt)^3s^{-5/3} 
 h\xxi\xpar{ x^{4/3} s\qqa e^{\ii\gt/3}},
\\
\phi_1(s,\gt)&\=
\rho\bigpar{1+\ii\tan\gt}-\rho\qqc s\qw e^{\ii\gt}(\sec\gt)\qqc
-\gamxi s e^{\ii\gt/2}.
\label{phi1}
\end{align}
In particular,
\begin{equation}\label{rephi}
\Re\phi_1(s,\gt)=
\rho-\rho\qqc s\qw (\cos\gt)\qqw
-\gamxi s \cos(\gt/2).
\end{equation}

In the cases
$\Xi\in\set{\xbm,\,\xdm,\,\xbmp}$ when $\nu=1$, we obtain similarly
\begin{equation}\label{fxib}
\fx\xxi(\glx)
=
\rho^{3/2}\glx^{7/3}
\int_{\gt=-\pi/2}^{\pi/2}\int_{s=0}^\infty
F_1(s,\gt;x)
e^{x^2\phi_1(s,\gt)}
\dd s \dd\gt
\end{equation}
where
\begin{align}
\label{F1b}
  F_1(s,\gt;x)&\=
\frac{1}{4\pi^2\ii}e^{\ii\gt/3}(\sec\gt)^{5/2}s^{-4/3} 
h\xxi\xpar{ x^{4/3} s\qqa e^{\ii\gt/3}}
\end{align}
and $\phi_1$ is the same as above.

Consider first $\gt=0$; then
\begin{align}
\phi_1(s,0)=
\Re\phi_1(s,0)=
\rho-\rho\qqc s\qw -\gamxi s,
\end{align}
which has a maximum at $s=s_0\=\rho^{3/4}\gamxi\qqw$.
In order for $(s_0,0)$ to be a saddle point of $\phi_1$, 
we need also
\begin{equation}\label{saddle}
0=
\frac{\partial\phi_1}{\partial\gt}(s_0,0)=
\ii \rho-\ii \rho\qqc s\qw -\half\ii\gamxi s
=\ii\bigpar{\rho-\tfrac32 \rho^{3/4}\gamxi\qq}
\end{equation}
and thus
\begin{equation}\label{rho}
  \rho=
  \rho\xxi\=
\Bigpar{\frac{3\gamxi\qq}2}^4=
\Bigpar{\frac{9\gamxi}4}^2
=
\begin{cases}
  9,& \Xi\in\set{\xbr,\xex,\xbrp},
\\
  9/4,& \Xi\in\set{\xbm,\xme,\xdm,\xbmp}.
\end{cases}
\end{equation}
With this choice of $\rho$, we find from \eqref{phi1} and
\eqref{saddle}
that the value
at the saddle point is
\begin{equation}\label{phi1b}
\phi_1(s_0,0)
=\rho-2 \rho^{3/4}\gamxi\qq
=-\frac \rho3
=
\begin{cases}
-  3,& \Xi\in\set{\xbr,\xex,\xbrp},
\\
 - 3/4,& \Xi\in\set{\xbm,\xme,\xdm,\xbmp}.
\end{cases}
\end{equation}
This yields the constant coefficient in the exponent of the
asymptotics.
We denote this value by $-b=-b\xxi$, and have thus, using \eqref{saddle},
\begin{align}\label{rhob}
  \rho&=3b,
&
\rho^{3/4}\gamxi\qq&=2b.
\end{align}
Further, by \eqref{rhob} and \eqref{phi1b}
\begin{align}\label{s0}
s_0
=
\rho^{3/4}\gamxi\qqw
=2b\gam\qw
=
\begin{cases}
 9/2,& \Xi\in\set{\xbr,\xex,\xbrp},
\\
 9/4,& \Xi\in\set{\xbm,\xme,\xdm,\xbmp}.
\end{cases}
\end{align}

The significant part of the integrals in \eqref{fxia} and \eqref{fxib}
comes from the square
\begin{equation}
  Q\=\bigset{(s,\gt):|s-s_0|\le \log x/x,\,|\gt|\le\log x/x}
\end{equation}
around the saddle point, as we will see in \refL{L2} below.
We consider first this square.

By \eqref{F1a}, \eqref{F1b} and \eqref{hbr}--\eqref{hbmp}, uniformly
for $(s,\gt)\in Q$, as \xtoo,
\begin{equation}\label{f1s}
F_1(s,\gt;x)
=
\begin{cases}
\frac{1+o(1)}{4\pi\qqc}s_0^{-5/3} 
\hh\xxi\xpar{x^{4/3}s_0\qqa},
& \Xi\in\set{\xbr,\xex,\xme,\xbrp},
\\[6pt]
\frac{1+o(1)}{4\pi^2}s_0^{-4/3} 
\hh\xxi\xpar{x^{4/3}s_0\qqa},
& \Xi\in\set{\xbm,\xdm,\xbmp}.
\end{cases}
\end{equation}
For the exponential part, we let $s=s_0(1+u/x)$ and $\gt=2v/x$, and
note that $Q$ corresponds to
\begin{equation}
  Q'\=\bigset{(u,v):|u|\le (\log x)/s_0,\,|\gt|\le(\log x)/2}.
\end{equation}
A Taylor expansion yields, for $(u,v)\in Q'$,
after straightforward computations,
\begin{equation}\label{phitaylor}
  \phi_1(s,\gt)
=-b-2bu^2x\qww +2\ii bu v x\qww -bv^2x\qww
+O\bigpar{(|u|^3+|v|^3)x^{-3}}.
\end{equation}
Hence,
\begin{equation}
  \begin{split}
  \iint_Q& e^{x^2\phi_1(s,\gt)}\dd s\dd\gt
=
2s_0 x\qww \iint_{Q'} e^{-b x^2-2bu^2+2\ii buv-bv^2+o(1)} \dd u\dd v
\\&
=
2s_0 x\qww e^{-b x^2}
\biggpar{\intoo\intoo e^{-2bu^2+2\ii buv-bv^2+o(1)} \dd u\dd v+o(1)}
\\
&=
2s_0 x\qww e^{-b x^2} 
\biggpar{
\pi
\begin{vmatrix}
  2b & -\ii b \\ 
-\ii b & b
\end{vmatrix}
\qqw
+o(1)}
\\&
\sim
\frac{2s_0\pi}{\sqrt3\,b} 
 x\qww e^{-b x^2} .
  \end{split}
\label{b2}
\end{equation}
Further, $\iint_Q \bigabs{e^{x^2\phi_1(s,\gt)}}\dd s\dd\gt$ is of
the same order. Consequently,
if we write
\begin{equation*}
G_1(s,\gt;x)\=
F_1(s,\gt;x)e^{x^2\phi_1(s,\gt)},
\end{equation*}
then \eqref{b2} and \eqref{f1s} yield
\begin{equation}\label{b3}
\iint_Q G_1(s,\gt;x)\dd s\dd\gt
=
\begin{cases}
\frac{1+o(1)}{2\sqrt{3\pi}\,b}s_0^{-2/3} 
\hh\xxi\xpar{x^{4/3}s_0\qqa}
x\qww e^{-b x^2},
& \Xi\in\set{\xbr,\xex,\xme,\xbrp},
\\[6pt]
\frac{1+o(1)}{2\sqrt3\,\pi b}s_0^{-1/3} 
\hh\xxi\xpar{x^{4/3}s_0\qqa}
x\qww e^{-b x^2},
& \Xi\in\set{\xbm,\xdm,\xbmp}.
\end{cases}
\end{equation}

For the complement $\Qc:=(0,\infty)\times(-\pi/2,\pi/2)\setminus Q$,
we have the following.

\begin{lemma}
  \label{L2}
For every $N<\infty$, for large $x$,
\begin{equation*}
  \iint_{\Qc} |G_1(s,\gt;x)|\dd s\dd\gt
= O\Bigpar{x^{-N} e^{-bx^2}}.
\end{equation*}
\end{lemma}

We postpone the proof and find from \eqref{fxia}, \eqref{fxib} and \eqref{b3},
using \eqref{rhob},
\begin{equation*}
\fx(x)
\sim
\begin{cases}
\frac{\sqrt3\,\rho}{2\sqrt{\pi}}s_0^{-2/3} 
\hh\xxi\xpar{x^{4/3}s_0\qqa}
x^{2/3} e^{-b x^2},
& \Xi\in\set{\xbr,\xex,\xme,\xbrp},
\\[6pt]
\frac{\sqrt{3\rho}}{2\pi}s_0^{-1/3} 
\hh\xxi\xpar{x^{4/3}s_0\qqa}
x^{1/3} e^{-b x^2},
& \Xi\in\set{\xbm,\xdm,\xbmp}.
\end{cases}
\end{equation*}
Substituting the functions $\hh\xxi$ implicit in \eqref{hbr}--\eqref{hbmp}
and the values of $\rho$, $b$ and $s_0$ given in
\eqref{rho}--\eqref{s0},
we finally find
\begin{alignat}2
\label{fxbr}
\fx\br(z)
&\sim 
\frac{\sqrt3\,\rho}{\sqrt{\pi}}s_0^{-1} 
 e^{-b x^2}
&&=
\frac{2\sqrt3}{\sqrt{\pi}}  e^{-3 x^2},
\\ 
\label{fxex}
\fx\ex(z)
&\sim
\frac{4\sqrt3\,\rho}{\sqrt{\pi}} x^{2} e^{-b x^2}
&&=
\frac{36\sqrt3}{\sqrt{\pi}} x^{2} e^{-3 x^2},
\\
\label{fxbm}
\fx\bm(z)
&\sim
\frac{\sqrt{3\rho}}{\sqrt\pi}s_0^{-1/2} 
e^{-b x^2}
&&=
\frac{\sqrt{3}}{\sqrt\pi}
e^{-3 x^2/4},
\\
\label{fxme}
\fx\me(z)
&\sim
{\sqrt3\,\rho}s_0^{-1/2} x e^{-b x^2}
&&=
\frac{3\sqrt3}2 x e^{-3 x^2/4},
\\
\label{fxdm}
\fx\dm(z)
&\sim
2\frac{\sqrt{3\rho}}{\sqrt\pi}s_0^{-1/2} 
e^{-b x^2}
&&=
\frac{2\sqrt{3}}{\sqrt\pi}
e^{-3 x^2/4},
\\
\label{fxbrp}
\fx\brp(z)
&\sim
\frac{\sqrt3\,\rho}{2\sqrt{\pi}}s_0^{-1} 
 e^{-b x^2}
&&=
\frac{\sqrt3}{\sqrt{\pi}}
 e^{-3 x^2},
\\
\label{fxbmp}
\fx\bmp(z)
&\sim
\frac{\sqrt{3\rho}}{2\sqrt\pi}s_0^{-1/2} 
 e^{-b x^2}
&&=\frac{\sqrt{3}}{2\sqrt\pi} e^{-3 x^2/4}
.
\end{alignat}

Recall that these are the densities of $\sqrtt\cB\xxi$. The density of
$\cB\xxi$ is 
$f\xxi(x)=\sqrtt\fx\xxi(\sqrtt x)$, and we obtain the 
leading term of the
asymptotics in Theorems \ref{Tex}--\ref{Tbmp}.
The leading terms of the asymptotics for $\P(\cB\xxi>x)$ follow by
integration by parts, as discussed in \refS{Stails}. 

It remains to prove \refL{L2}.
We begin by observing that by
\eqref{Hx0}, \eqref{psih} and
\eqref{hbr}--\eqref{hbmp},
\begin{equation}
  |h(z)|=O\bigpar{|z|+|z|\qw},
\qquad \aaz<\pi/6.
\end{equation}
Hence \eqref{F1a} and \eqref{F1b} show that, with some margin,
\begin{equation}
  |F_1(s,\gt;x)| \le \CC\bigpar{x^2s\qw+x\qww s^{-3}}(\cos\gt)^{-3}.
\end{equation}
and thus 
by \eqref{rephi},
for $x\ge1$,
\begin{equation}\label{b1}
  |G_1(s,\gt;x)| 
\le 
\CC
x^2(\cos\gt)^{-3}
\bigpar{s\qw+ s^{-3}}
e^{\rho x^2-x^2A(\gt)s\qw-x^2B(\gt)s},
\end{equation}
where $A(\gt)=\rho\qqc(\cos\gt)\qqw$ and $B(\gt)\=\gam\cos(\gt/2)$.
We integrate over $s$, using the following lemma.

\begin{lemma}
  \label{L1}
Let $M\ge0$. 
\begin{thmenumerate}
  \item
If $A$ and $B$ are positive numbers and $AB\ge1$, then
\begin{equation}
  \intoo s^{-M-1}e^{-As\qw-Bs}\dd s \le \CC(M)(B/A)^{M/2} e^{-2\sqrt{AB}}.
\end{equation}
\item
If further $0<\gd<1$, then
\begin{equation}
  \int_{{\bigabs{s-\sqrt{A/B}}>\gd\sqrt{A/B},\,s>0}} 
s^{-M-1}e^{-As\qw-Bs}\dd s \le \CC(M)(B/A)^{M/2} e^{-(2+\gd^2/2)\sqrt{AB}}.
\end{equation}
\end{thmenumerate}
\end{lemma}

\begin{proof}
\pfitem{i}
  The change of variables $s=\sqrt{A/B}\,t$ followed by $t\mapsto
  t\qw$ for $t>1$ yields
\begin{equation}\label{c2}
  \begin{split}
  \intoo s^{-M}e^{-As\qw-Bs}\frac{\dd s}s
&=(B/A)^{M/2}  \intoo t^{-M}e^{-\sqrt{AB}(t\qw+t)}\frac{\dd t}t
\\&
=(B/A)^{M/2}  \intoi \bigpar{t^{-M}+t^M}e^{-\sqrt{AB}(t\qw+t)}\frac{\dd t}t	
  \end{split}
.
\end{equation}
 
For $t\in(\frac16,1)$ we write $t=1-u$ and use $(1-u)\qw+1-u\ge2+u^2$;
hence the integral over $(\frac16,1)$ is bounded by
\begin{equation*}
 \CC(M) \intoo e^{-\sqrt{AB}(2+u^2)}\dd u
\le \CC(M)e^{-2\sqrt{AB}}
\end{equation*}

For $t\in(0,\frac16)$ we use
\begin{equation*}
 t^{-M-1}e^{-\sqrt{AB}\,t\qw/2}
\le \CC(M) (AB)^{-(M+1)/2}
\le \CCx(M);
\end{equation*}
hence the integral over $(0,\frac16)$ is bounded by
\begin{equation}\label{c1}
\CCx(M) 
\int_0^{1/6}
e^{-\sqrt{AB}\,t\qw/2}\dd t
\le 
\CCx(M) e^{-3\sqrt{AB}}.
\end{equation}

\pfitem{ii}
Arguing as in \eqref{c2}, we see that the integral is bounded by 
\begin{equation*}
  \begin{split}
(B/A)^{M/2}  \int_0^{1/(1+\gd)} 
 2t^{-M}e^{-\sqrt{AB}(t\qw+t)}\frac{\dd t}t	
  \end{split}
.
\end{equation*}
The integral over $(0,1/6)$ is bounded by \eqref{c1}, and the integral
over $(1/6,\\1/(1+\gd))$ by 
\begin{equation*}
  \begin{split}
\CC(M)e^{-\sqrt{AB}\bigpar{1+\gd+1/(1+\gd)}}
\le
\CCx(M)e^{-\sqrt{AB}\xpar{2+\gd^2/2}} .
 \qedhere
  \end{split}
\end{equation*}

\end{proof}

\begin{proof}[Proof of \refL{L2}]
Returning to \eqref{b1},
we have $B(\gt)/A(\gt)\le\gam/\rho\qqa$ and
\begin{equation}\label{AB}
  A(\gt)B(\gt)=\rho\qqa\gam (\cos\gt)\qqw\cos(\gt/2).
\end{equation}
Noting that $\rho\qqa\gam=(2b)^2$ by \eqref{rhob} and
\begin{equation*}
  \frac{\cos(\gt/2)^2}{\cos\gt}
=  \frac{\cos\gt+1}{2\cos\gt}
=\frac12+  \frac{1}{2\cos\gt}
\ge 1+\cc\gt^2,
\qquad |\gt|<\pi/2,
\end{equation*}
we see that
\begin{equation}\label{AB2}
\sqrt{A(\gt)B(\gt)}\ge2b+\cc\gt^2.
\end{equation}
Hence \refL{L1} applies with $A=x^2A(\gt)$ and $B=x^2B(\gt)$ when
$x^2\ge 1/(2b)$
and shows, using \eqref{b1}, that for every $\gt$ with
$|\gt|<\pi/2$, 
\begin{equation}\label{l1a}
  \begin{split}
\int_{0}^\infty
|G_1(s,\gt;x)|\dd s 
&
\le \CC x^2(\cos\gt)^{-3} e^{\rho x^2 -4b x^2-2\ccx x^2\gt^2}
\\&
= \CCx x^2(\cos\gt)^{-3} e^{-b x^2-\cc x^2\gt^2}
.
  \end{split}
\end{equation}
For $|\gt|$ close to $\pi/2$, we use instead of \eqref{AB2} 
\begin{equation}\label{AB3}
\sqrt{A(\gt)B(\gt)}\ge\cc(\cos\gt)\qqqqw,
\end{equation}
another consequence of \eqref{AB}. Hence, \eqref{b1} and \refL{L1}(i)
show that if $\eps>0$ is small enough, and $|\gt|>\pi/2-\eps$, then
\begin{equation}\label{l1b}
  \begin{split}
\int_{0}^\infty
|G_1(s,\gt;x)|\dd s 
&
\le \CC x^2(\cos\gt)^{-3} e^{\rho x^2 -\ccx x^2(\cos\gt)\qqqqw}
\\&
\le \CC e^{-b x^2-\ccx x^2/2}
.
  \end{split}
\end{equation}
  Moreover, \eqref{AB} implies that if $|\gt|\le1$, say,
then
\begin{equation*}
  \sqrt{A(\gt)/B(\gt)} = s_0+O(\gt^2).
\end{equation*}
Hence, if $|\gt|\le(\log x)/x$ and 
$|s-s_0|>(\log x)/x$, then, for large $x$,
\begin{equation*}
\lrabs{s- \sqrt{A(\gt)/B(\gt)}}
>\frac{\log x}{2x}
>\cc\frac{\log x}{x} \sqrt{A(\gt)/B(\gt)},
\end{equation*}
and \refL{L1}(ii) implies, using \eqref{AB2}, that 
if $|\gt|\le(\log x)/x$, then
\begin{equation}\label{l1c}
  \begin{split}
\int_{|s-s_0|>\log x/x,\,s>0}
|G_1(s,\gt;x)|\dd s 
&
\le \CC x^2 e^{\rho x^2 -4bx^2-\cc (\log x)^2}
\\&
\le \CC  e^{-bx^2-\cc (\log x)^2}
.
  \end{split}
\end{equation} 
The lemma follows by using  \eqref{l1c} for $|\gt|\le(\log x)/x$, 
\eqref{l1b} for $|\gt|>\pi/2-\eps$,
and \eqref{l1a} for the remaining $\gt$, and integrating over $\gt$.
\end{proof}

\section{Higher order terms}\label{Shigher}

The asymptotics for $f\xxi(x)$ obtained above can be refined to full
asymptotic expansions by standard methods and straightforward, but
tedious, calculations. 
With possible future extensions in view, we find it instructive to
present two versions of this; the first is more straightforward brute
force, while the second (in the next section)
performs a change of variables leading to
simpler integrals.

First, the asymptotics \eqref{Bioo} and \eqref{Biioo} can be refined
into well-known asymptotic series \cite[10.4.63,10.4.66 (with a typo
  in early printings)]{AS}; we write these as
\begin{align}\label{Biasymp}
\Bi(z)&=
\pi^{-1/2}z^{-1/4}e^{\zetax}\gb_0(z),
\\
\Bi'(z)&=
\pi^{-1/2}z^{1/4}e^{\zetax}\gb_1(z),
\label{Biiasymp}
\end{align}
with
\begin{align}\label{gb0}
\gb_0(z)&
={ 1+\frac{5}{48} z^{-3/2} +\frac{385}{4608} z^{-3}+\dots
+O\bigpar{z^{-3N/2}}},
&&|\arg(z)|\le\pi/3-\gd,
\\ \label{gb1}
\gb_1(z)&
={ 1-\frac{7}{48} z^{-3/2} -\frac{455}{4608} z^{-3}+\dots
+O\bigpar{z^{-3N/2}}},
&&|\arg(z)|\le\pi/3-\gd,
\end{align}
where the expansions can be continued to any desired power $N$ of
$z\qqc$.
Similarly, 
\eqref{BIoo} can be refined to an asymptotic series
\begin{align}\label{BIasymp}
\BI(z)&=
\pi^{-1/2}z^{-3/4}e^{\zetax}\gb_{-1}(z),
\end{align}
with
\begin{align}\label{gb-1}
\gb_{-1}(z)&
={ 1+\frac{41}{48} z^{-3/2} +\frac{9241}{4608} z^{-3}+\dots
+O\bigpar{z^{-3N/2}}},
&&|\arg(z)|\le\pi/3-\gd;
\end{align}
this is easily verified by writing \eqref{BI} as
$\BI(z)=\BI(1)+\int_1^z t\qw\Bi''(t)\dd t$ followed by repeated
integrations by parts, as in corresponding argument for $\AI(z)$ in
\cite[Appendix A]{SJ201}. The coefficients in \eqref{gb-1}
are easily found noting that
a formal differentiation of \eqref{BIasymp} yields \eqref{Biasymp}.
(They are the numbers denoted $\gb_k$ in \cite{SJ201}.)

Hence, by \eqref{psixbr}--\eqref{psixbmp}, \eqref{hhxi} can be refined
to, for $\aaz\le\pi/6$,
\begin{equation}\label{hhhxi}
  h\xxi(z)=\ii\hh\xxi(z)\Bigpar{\hhh\xxi(z)+O\lrpar{e^{-2z\qqc/3}}},
\end{equation}
where 
\begin{alignat}2
\hhh\br(z)&\=\gb_1(z)\qww
&&=
 1+\frac{7}{24} z^{-3/2} +\dots
,
 \\
\hhh\ex(z)&\=\gb_1(z)\gb_0(z)^{-3}
&&=
 1-\frac{11}{24} z^{-3/2} +\dots
,                                                \label{hexcurs}
 \\
\hhh\bm(z)&\=\gb_1(z)\qw
&&=
 1+\frac{7}{48} z^{-3/2} +\dots
,
 \\
\hhh\me(z)&\=\gb_0(z)\qw
&&=
 1-\frac{5}{48} z^{-3/2} +\dots
,
 \\
\hhh\dm(z)&\=\gb_{-1}\gb_0(z)\qww
&&=
 1+\frac{31}{48} z^{-3/2} +\dots
,
 \\
\hhh\brp(z)&\=\bigpar{(\gb_0(z)+\gb_1(z))/2}\qww
&&=
 1+\frac{1}{24} z^{-3/2} +\dots
 ,
\\
\hhh\bmp(z)&\=\bigpar{(\gb_0(z)+\gb_1(z))/2}\qw
&&=
 1+\frac{1}{48} z^{-3/2} +\dots
\end{alignat}
By \eqref{gb0} and \eqref{gb1}, $\hhh\xxi$ has an asymptotic series
expansion
\begin{equation}
\label{hhh1}
\hhh\xxi(z)
= 1+d\xxix1 z^{-3/2} +d\xxix2 z^{-3}+\dots
+O\bigpar{z^{-3N/2}},
\qquad |\arg(z)|\le\pi/3-\gd,
\end{equation}
for some readily computed coefficients $d\xxix{k}$; moreover, we can
clearly ignore the $O$ term in \eqref{hhhxi}.

Next, by \refL{L2}, it suffices to consider $(s,\gt)\in Q$ in
\eqref{fxia} and \eqref{fxib}. We use \eqref{hhhxi} and \eqref{hhh1}
in \eqref{F1a} and \eqref{F1b} and obtain, for example, for $\Xi=\xex$,
\begin{equation}
  F_1(s,\gt;x)
=
\frac{2 x^{4/3}}{\pi\qqc}\Bigpar{e^{i\gt}(\sec\gt)^3s\qw
-\frac{11}{24}e^{i\gt/2}(\sec\gt)^3s\qww x\qww+\dots+O(x^{-2N})}.
\end{equation}
We substitute $s=s_0(1+u/x)$ and $\gt=2v/x$ as above and obtain by
Taylor expansions a series in $x\qw$ (with a prefactor $x^{4/3}$)
where the coefficients are polynomials in $u$ and $v$.

Similarly, the Taylor expansion \eqref{phitaylor}
can be continued; we write the remainder term as $R(u,v;x)$ and have 
\begin{equation}
  R(u,v;x)=r_3(u,v)x^{-3}+\dots +O(x^{-2N-2}),
\end{equation}
for some polynomials $r_k(u,v)$. Another Taylor expansion then yields
\begin{equation}
e^{x^2  R(u,v;x)}=r^*_1(u,v)x^{-1}+\dots +O(x^{-2N}),
\end{equation}
for some polynomials $r^*_k(u,v)$. We multiply this, 
the expansion of $F_1$
and the main term
$\exp\xpar{-bx^2-2bu^2+2\ii buv-bv^2}$, 
and integrate over $Q$; we may extend the integration domain
to $\bbR^2$ with a negligible error.
This yields an asymptotic expansion for $\fx_\Xi(x)$, and thus for $f_\Xi(x)$,
where the leading term found above is multiplied by a series in $x\qw$, up to
any desired power. Furthermore, it is easily seen that all coefficients
for odd powers of $x\qw$ vanish, since they are given by the integrals
of an odd functions of $u$ and $v$; hence this is really an asymptotic
series in $x\qww$.

We obtain
the explicit expansions for $f_\Xi(x)$ in 
Theorems \ref{Tex}--\ref{Tbmp}
by calculations with \maple{}.
The asymptotics for $\P(\cB\xxi>x)$ follow by
integration by parts, see \refS{Stails}. 

\begin{remark}
  \label{R++}
In particular, since $\hh\brp=\half\hh\br$
and $\hh\bmp=\half\hh\bm$, the leading terms 
for $\xbrp$ and $\xbmp$
differ from those of $\xbr$ and $\xbm$ by a factor $\half$
as discussed in \refR{R+}.
The second order terms in $\hhh$ are different, as is seen above; more
precisely, $\hhh\brp=\hhh\br-\frac14z\qqcw+O(z^{-3})$ and
$\hhh\bmp=\hhh\bm-\frac18z\qqcw+O(z^{-3})$; it is easily seen that
if we ignore terms beyond the second, this difference transfers into 
factors $1-(4s_0)\qw x\qww$ and 
$1-(8s_0)\qw x\qww$, respectively, for $\fx\xxi$, which in both cases
equals $1-\frac1{18}x\qww$, and thus a factor $1-\frac1{36}x\qww$ for
$f\xxi$, which explains the difference between the second order terms
in $f\br$ or $f\bm$ and $2f\brp$ or $2f\bmp$; \cf{} again \refR{R+}.
\end{remark}

\section{Higher order terms, version II}\label{Shigher2}

Our second version of the saddle point method leads to simpler
calculations (see for
instance, Bleistein and Handelsman \cite{BH86}). We illustrate it with
$\cBex$; the other Brownian areas 
are treated similarly. We use again \eqref{fxia}, and recall that by
\refL{L2}, it suffices to consider $(s,\gt)$ close to
$(s_0,0)=(9/2,0)$.

We make first the substitution $s=\frac92 (\sec\gt)\qqc u\qw$
(this is not necessary, but makes the integral more similar to
Tolmatz' versions).
This transforms \eqref{fxia} into
\begin{equation}\label{ex1}
\fx\ex(\glx)
=\int_{\tet=-\pi/2}^{\pi/2}\int_{u=0}^\II F_{2}(u,\tet;x)
 \expQ{x^2\Fi_2(u,\tet)} \dd u \dd\tet,
\end{equation}
where, by \eqref{F1a}, \eqref{phi1}, \eqref{rho}, \eqref{hhhxi}, \eqref{hex},
\eqref{hexcurs}, for $u$ bounded, at least, 
\begin{align}
  F_2(u,\gt;x)&=
\frac{81 x^{8/3}}{4\pi\qqc\ii}e^{2\ii\gt/3}(\sec\gt)^2
\lrpar{\tfrac92}\qqaw 
u^{-1/3} 
 h\ex\Bigpar{ \bigpar{\tfrac92}\qqa x^{4/3} u\qqaw  e^{\ii\gt/3}\sec\gt},
\nonumber
\\ \label{F2}
&=
\frac{162 x^{4}e^{\ii\gt}}{\pi\qqc u (\cos\gt)^3}
 \hhh\ex\Bigpar{ \bigpar{\tfrac92}\qqa x^{4/3} u\qqaw
   e^{\ii\gt/3}\sec\gt}
\\
&=\frac{162 x^4 e^{i\tet}}{\pi^{3/2}u (\cos \tet)^3}
-\frac{33  x^2  e^{\ii\tet/2}}{2\pi^{3/2} (\cos \tet)^{3/2}}
+O(1)
, \label{F2a}
\\
\Fi_2(u,\tet)&=  9(1+\ii\tan\tet)-6 u e^{\ii\tet}
-\frac{6e^{-\ii\tet}(1+\ii\tan\tet)^{3/2}}{u}
.\label{phi2}
\end{align}
The saddle point is now $(u,\gt)=(1,0)$, and in a neighbourhood we
have, \cf{} \eqref{phitaylor}, with $\xu=u-1$,
\begin{equation}\label{phi2a}
  \phi_2(u,\gt)
=-3-6 \xu^2 -3\ii \xu \gt -\tfrac34\gt^2
+O\bigpar{|\xu|^3+|\gt|^3}.
\end{equation}

The function $\Fi_2$ has a non-degenerate critical point at $(1,0)$,
and by the Morse lemma, see \eg{} \citet[Lemma 2.2]{Milnor},
we can make a complex analytic change of variables in a neighbourhood
of $(1,0)$ such that in the new variables $\Fi_2+3$ becomes a diagonal
quadratic form. (The Morse lemma is usually stated for real variables,
but the standard proof in \eg{} \cite{Milnor} applies to the complex
case too.)
The quadratic part of \eqref{phi2a} is diagonalized by $(\tv,\gt)$ with 
$v=\tv-\ii\gt/4$; we may thus choose the new variables $\tu$ and
$\tgt$ 
such that $\tu\sim\tv$ and $\tgt\sim\gt$ at the critical point, and
thus
\begin{align}
  u&=
1+\tu-\ii\tgt/4+O\bigpar{|\tu|^2+|\tgt|^2}, \label{mo1}
\\
  \gt&=
\tgt+O\bigpar{|\tu|^2+|\tgt|^2}, \label{mo2}
\\
\Fi_2(u,\gt)&=-3-6\tu^2-\tfrac98\tgt^2. \label{phi2b}
\end{align}
Note that the new coordinates are not uniquely determined; we will
later use this and simplify by letting some Taylor coefficients be 0.
In the new coordinates,  \eqref{ex1} yields, 
\begin{equation}\label{ex2}
\fx\ex(\glx)
\sim\int_{\ttet}\int_{\tu} F_{3}(\tu,\ttet;x)
 \expQ{-3x^2-x^2(6\tu^2+\frac98\ttet^2)} 
J(\tu,\tgt)
\dd \tu \dd\ttet,
\end{equation}
where 
$F_3$ is obtained by
substituting $u=u(\tu,\tgt)$ and $\gt=\gt(\tu,\tgt)$ in \eqref{F2}
and
$J(\tu,\tgt)
=\frac{\partial u}{\partial \tu}\frac{\partial \gt}{\partial \tgt}
-\frac{\partial u}{\partial \tgt}\frac{\partial \gt}{\partial \tu}$
is the Jacobian. 
Recall that, up to a negligible error, we only have to integrate in
\eqref{ex1} over a small disc, say with radius $\log x/x$; this becomes
in the new coordinates a surface in $\bbC^2$ as the integration domain
in \eqref{ex2}. The next step is to replace this integration domain
by, for example, the 
disc \set{(\tu,\tgt)\in\bbR^2:|\tu|^2+|\tgt|^2\le (\log x/x)^2}, 
in analogy with the much more standard change of integration
contour in one complex variable. To verify the change of integration
domain, note that if $F(z_1,z_2)$ is any analytic function of two
complex variables, then $F(z_1,z_2)\dd z_1\wedge\dd z_2$ is a closed
differential form in $\bbC^2$ (regarded as a real manifold of
dimension four), and thus
$\int_{\partial M} F(z_1,z_2)\dd z_1\wedge\dd z_2=0$  by Stokes'
theorem for any compact submanifold $M$ with boundary $\partial M$.
In our case, it follows that the difference between the integrals over
the two domains equals an integral over boundary terms at a distance
$\asymp \log x/x$ from the origin, which is negligible.
(The careful reader may parametrize the two domains by suitable mappings
$\psi_0,\psi_1:U\to\bbC^2$, where $U$ is the unit disc in $\bbR^2$,
and apply Stokes' theorem to the cylinder $U\times[0,1]$ and the
pullback of $F(z_1,z_2)\dd z_1\wedge\dd z_2$ by the map 
$(w,t)\mapsto (1-t)\psi_0(w)+t\psi_1(w)$.) 

We next change variable again to $\qr=x\tu$, $t=x\tgt$, and obtain by
\eqref{ex2} 
\begin{equation}\label{ex3}
\fx\ex(\glx)
\sim
x\qww e^{-3x^2}
\iint F_{3}(\qr/x,t/x;x)
J(\qr/x,t/x)
 \expQ{-6\qr^2-\frac98t^2} 
\dd \qr \dd t,
\end{equation}
integrating over $(\qr,t)\in\bbR^2$ with, say, $\qr^2+t^2\le(\log x)^2$.
To obtain the desired asymptotics for $\fx\ex$, and thus for $f\ex$,
we mechanically expand $F_3$ and $J$ in Taylor series up to any
desired order and compute the resulting Gaussian integrals, extending
the integration domains to $\bbR^2$.

We illustrate this by giving the details for the first two terms in
\eqref{tex1}. 
We have, \cf{} \eqref{mo1} and \eqref{mo2}, expansions
\begin{align*}
u&=1+\xr-\ii \xs/4+\al_1 \xr^2+\al_2 \xr\xs+O(|\xr|^3+|\xs|^3),\\
\tet&=\xs+\al_3 \xs^2+\al_4 \xr\xs+O(|\xr|^3+|\xs|^3),
\end{align*}
where we, as we may, have chosen two Taylor coefficients to be 0.
To determine $\al_1,\dots,\al_4$, we substitute into
$\Fi_2(u,\tet)$. We obtain from \eqref{phi2}, up to terms of order three,
\begin{multline*} 
\Fi_2(u,\tet)\sim -3-[6\xr^2+9\xs^2/8]
+ (6-12\al_1)\xr^3+(-3\ii\al_4-12\al_2-15\ii/2)\xs\xr^2\\
+ \lp -\frac{9\al_4}{4}-3\ii\al_3+\frac{33}{8}\rp \xs^2\xr
+\lp -\frac{9\al_3}{4}+\frac{15\ii}{32}\rp \xs^3.
\end{multline*}
Annihilating the coefficients, \cf{} \eqref{phi2b}, leads to a linear
system, the solution 
of which is 
\begin{align*}
 & \al_1=1/2,&&\al_2=-83\ii/72,&&\al_3=5\ii/24,&&\al_4=19/9.
\end{align*}

This leads to the Jacobian
\beqs J(\xr,\xs)
=  1+ \lp -\frac{5 \ii \xs}{24}+\frac{28 \xr}{9}\rp 
+O\bigpar{\xs^2+\xr^2}
.\eeqs
Furthermore, by \eqref{F2a}, with $\qr=x\tu$ and $t=x\ttet$, 
\beqs 
F_3(\xr,\xs;x)=
F_{2}(u,\tet;x)
\sim
\frac{162 x^4}{\pi^{3/2}}+\frac{162
  x^3}{\pi^{3/2}}\Bigpar{\frac54\ii t-\qr}
+O\bigpar{x^2(1+\qr^2+t^2)}.
\eeqs
Integrating in \eqref{ex3} yields the leading term
\begin{equation}
  \label{E4}
\fx\ex(x)\sim \frac{36 \sqrt3}{\pi^{1/2}}\,x^2 e^{-3 x^2}
\end{equation}
together with correction terms of order $xe^{-3 x^2}$ that all
vanish by symmetry, since they involve integrals of odd functions,
plus a remainder term of order $e^{-3 x^2}$.

The next term in the expansion of $e^{3x^2}\fx\ex$ is thus the constant term.
To find it,  we try, again setting some Taylor coefficients to 0 as we
may,
\bals
u&\sim1+(\xr-\ii \xs/4)+\xr(\al_1 \xr+\al_2 \xs)+\xr(\be_1 \xr^2+\be_2 \xr\xs+\be_3 \xs^2),\\
\tet&\sim \xs+\xs(\al_3 \xs+\al_4 \xr)+\xs(\be_4 \xr^2+\be_5 \xr\xs+\be_6 \xs^2).
\end{align*}
We obtain now
\begin{multline*}
\Fi_2(u,\tet)\sim -3-[6\xr^2+9\xs^2/8]
+(3/2-12\be_1)\xr^4+(-131\ii/6-3\ii\be_4-12\be_2)\xr^3\xs
\\
+(-9\be_4/4+2627/288-12\be_3-3\ii\be_5)\xr^2\xs^2 
+(-9\be_5/4+535\ii/96-3\ii\be_6)\xr\xs^3 
\\
+(-1283/768-9\be_6/4)\xs^4.
\end{multline*}
We set for instance $\be_4=0$. This gives
\beqs \be_1=1/8
,\be_2=-131\ii/72,\be_3=16867/10368,\be_5=4493\ii/1296,\be_6=-1283/1728.\eeqs 
The Jacobian becomes
\beqs J(\xr,\xs)\sim  1+ \lp \frac{28 \xr}{9}-\frac{5 \ii \xs}{24}\rp+ \lp
\frac{179}{72}\xr^2+\frac{2405}{648}\ii \xr\xs -\frac{379}{384}\xs^2\rp.\eeqs 
The first term in \eqref{F2a} becomes
\beqs 
\sim\frac{162 x^4}{\pi^{3/2}}
+\frac{162 x^4}{\pi^{3/2}}\Bigpar{\frac54\ii \xs-\xr}
+\frac{x^4}{\pi^{3/2}}\Bigpar{81 \xr^2+\frac{1143}4 \ii \xr\xs+\frac{621}8 \xs^2}.\eeqs
and the second is
\beqs 
\sim-\frac{33 x^2}{2\pi^{3/2}}
.\eeqs
Collecting terms, the coefficient of $x^2$ in 
$ F_{3}(\qr/x,t/x)J(\qr/x,t/x)$ equals
\beqs -\frac{1296 \qr^2-99248 \ii \qr t+2565t^2}{64 \pi^{3/2}}
-\frac{33}{2\pi^{3/2}}
.\eeqs
Multiplying by $\exp(-6\qr^2-\frac98 t^2)$ and integrating 
yields the contribution
\beq
x^2\frac{-8\sqrt3}{\pi^{1/2}}  \label{E5}
\eeq
to the integral in \eqref{ex3}.
So, finally, combining (\ref{E4}) and (\ref{E5}),
\beqs \fx\ex(x)\sim \frac{3^{1/2}e^{-3 x^2}}{\pi^{1/2}}\lb 36 x^2 - 8 \rb,\eeqs
which fits with the first two terms for $f\ex(x)$ in \refT{Tex}.
More terms can be found in a mechanical way. 

\section{Proof of \refT{P1}}\label{SproofP1}
  Let $T\sim\Gamma(\nu)$ be a Gamma distributed random variable
independent of $X$ and let $X_T\=T\qqc X$.
Then $T$ has the density $\Gamma(\nu)\qw t^{\nu-1}e^{-t}$, $t>0$, and
thus $X_T$ has, 
using \eqref{a1}, the Laplace transform 
\begin{equation}
  \label{a2}
  \begin{split}
\psi_T(u)
&:=\E e^{-uT\qqc X}
=\E\psi\bigpar{u T\qqc}
\\&
=\Gamma(\nu)\qw\intoo\psi\bigpar{ut\qqc}t^{\nu-1}e^{-t}\dd t
\\&
=\Gamma(\nu)\qw\intoo\psi\bigpar{s\qqc}u^{-2\nu/3}s^{\nu-1}e^{-u\qqaw s}\dd s
\\&
=u^{-2\nu/3}\Psi\bigpar{u\qqaw},
\qquad u>0.	
  \end{split}
\end{equation}
By \eqref{a2} and our assumption on $\Psi$, $\psi_T$ extends to an
analytic function 
in $\bbC\setminus(\infty,0]$.
Furthermore, $X_T$ has a  density $g$ on $(0,\infty)$,
because $T\qqc$ has, and it is easily verified that this density is
continuous.
We next use Laplace inversion for $X_T$. The Laplace transform
$\psi_T$ is, by \eqref{a2}, \emph{not} absolutely integrable on
vertical lines
(at least not in our cases, where $\Psi(z)$ is bounded away from $0$
as $z\to0$),
so we will use the following form of the Laplace inversion formula,
assuming only conditional convergence of the integral.
\begin{lemma}
  \label{LLinv}
Let $\fh$ be a measurable function on $\bbR$.
Suppose that the Laplace transform $\tfh(z)\=\intoooo
\fh(y)e^{-zy}\dd y$ exists in a strip $a<\Re z<b$, and that
$\gs\in(a,b)$ is a real number such that the 
generalized integral 
$\int_{\gs-\ii\infty}^{\gs+\ii\infty} e^{xz}\tfh(z)\dd z$
exists in the sense that the limit $\lim_{A\to\infty} s_A$ exists, where
\begin{equation*}
s_A
\=
 \int_{\gs-\ii A}^{\gs+\ii A} e^{xz}\tfh(z)\dd z.
\end{equation*}
If further $x$ is a continuity point
(or, more generally, a Lebesgue point)
of
$\fh$, then 
\begin{equation*}
  \int_{\gs-\ii\infty}^{\gs+\ii\infty} e^{xz}\tfh(z)\dd z
\=
\lim_{A\to\infty} s_A
=2\pi \ii \fh(x).
\end{equation*}
\end{lemma}

\begin{proof}
  By considering instead $e^{-\gs y}\fh(y)$, we may suppose that
  $\gs=0$. In this case, 
$h$ is integrable and
$\tfh(\ii t)=\hfh(t)$, the Fourier transform of
  $\fh$, and the result is a classical result on Fourier inversion. (It
  is the analogue for Fourier transforms of the more well-known fact
  that if a Fourier series converges at a continuity 
(or Lebesgue)
point of the
  function, then the limit equals the function value.)
For a proof, note that if $s_A$ converges as $A\to\infty$, then so
  does the Abel mean
$\intoo ye^{-yA}s_A\dd A$ as $y\to0$, and this Abel mean
equals $2\pi\ii$ times the Poisson integral 
$\intoooo\pi\qw y(u^2+y^2)\qw \fh(x-u)\dd u$, which converges to $\fh(x)$.
\end{proof}

We verify the condition of the lemma with $h=g$ and $\gs=1$, recalling
that $\tg=\psi_T$. 
Thus, by \eqref{a2},
\begin{equation*}
s_A
\=
 \int_{1-\ii A}^{1+\ii A} e^{xz}\psi_T(z)\dd z
=
 \int_{1-\ii A}^{1+\ii A} e^{xz}z^{-2\nu/3}\Psi\xpar{z\qqaw}\dd z.
\end{equation*}
We may here change the integration path from the straight line segment
$[1-\ii A,1+\ii A]$ to the path consisting of the following seven
parts:
\begin{itemize}
  \item[$\gamma_1$:] the line segment $[1-\ii A,-A-\ii A]$,
  \item[$\gamma_2$:] the line segment $[-A-\ii A,-A-\ii 0]$,
  \item[$\gamma_3$:] the line segment $[-A-\ii 0,-\eps-\ii 0]$,
  \item[$\gamma_4$:] the circle \set{\eps e^{\ii\phit}: \phit\in[-\pi,\pi]}.
  \item[$\gamma_5$:] the line segment $[-\eps+\ii 0,-A+\ii 0]$,
  \item[$\gamma_6$:] the line segment $[-A+\ii 0,-A+\ii A]$,
  \item[$\gamma_7$:] the line segment $[-A+\ii A,1+\ii A]$.
\end{itemize}
(Here, $\gamma_3$ could formally be interpreted as the line segment
$[-A-\ii\eta,-\sqrt{\eps^2-\eta^2}-\ii\eta]$ for a small positive
$\eta$, taking the limit of the integral
as $\eta\to0$, and similarly for the other
parts with $\pm\ii0$.)
Letting $A\to\infty$, we see that we essentially change the
integration path from a vertical line to a Hankel contour; however, we
do this carefully since, as said above, the integral along the
vertical line is not absolutely convergent.

We now first let $\eps\to0$. By \eqref{Hoo},
\begin{equation*}
   \int_{\gamma_4} e^{xz}z^{-2\nu/3}\Psi\xpar{z\qqaw}\dd z
=O\lrpar{\eps^{1-2\nu/3}} \to0,
\end{equation*}
and, again by \eqref{Hoo}, 
the integrals along $\gamma_3$ and
$\gamma_5$
converge to the absolutely convergent integrals
\begin{equation*}
  \int_{-A-\ii0}^{-\ii0} e^{xz}z^{-2\nu/3}\Psi\xpar{z\qqaw}\dd z
=\int_0^A e^{-x\rho}\rho^{-2\nu/3} e^{2\pi\nu\ii/3}
  \Psi\bigpar{e^{2\pi\ii/3}\rho\qqaw}\dd\rho
\end{equation*}
and
\begin{equation*}
  \int_{\ii0}^{-A+\ii0} e^{xz}z^{-2\nu/3}\Psi\xpar{z\qqaw}\dd z
=-\int_0^A e^{-x\rho}\rho^{-2\nu/3} e^{-2\pi\nu\ii/3}
  \Psi\bigpar{e^{-2\pi\ii/3}\rho\qqaw}\dd\rho,
\end{equation*}
which together make
\begin{equation*}
I_A:=\int_0^A e^{-x\rho}\rho^{-2\nu/3} \Hx\bigpar{\rho\qqaw}\dd\rho.
\end{equation*}
Hence, for every $A>0$,
\begin{equation*}
  s_A=I_A+\lrpar{\int_{\gam_1}+\int_{\gam_2}+\int_{\gam_6}+\int_{\gam_7}}
e^{xz}z^{-2\nu/3}\Psi\xpar{z\qqaw}\dd z.
\end{equation*}
Now let $A\to\infty$.
By \eqref{H0},
\begin{equation*}
   \int_{\gamma_1} e^{xz}z^{-2\nu/3}\Psi\xpar{z\qqaw}\dd z
=o\lrpar{\int_{-\infty}^1 e^{xt}\dd t}=o(1),
\end{equation*}
and similarly $\int_{\gam_2}=o(1)$, $\int_{\gam_6}=o(1)$,
$\int_{\gam_7}=o(1)$. 
Finally, $I_A\to I_\infty$, and \refL{LLinv} applies and yields the following.

\begin{lemma}\label{Lg}
For every $x>0$, we have
  \begin{equation}\label{g4}
	g(x)=\frac{1}{2\pi\ii}\intoo
	e^{-x\rho}\rho^{-2\nu/3}\Hx\xpar{\rho\qqaw}\dd\rho,
  \end{equation}
where the integral is absolutely convergent by \eqref{Hx0} and \eqref{Hxoo}.
\end{lemma}

By the change of variables $\rho=u\qqcw$, \eqref{g4} may be rewritten as
  \begin{equation}
	g(x)=\frac{3}{4\pi\ii}\intoo
	e^{-x u\qqcw}u^{\nu-5/2}\Hx\xpar{u}\dd u,
\qquad x>0.
  \end{equation}
We can here, using \eqref{Hx0} and \eqref{Hxoo}, change the
integration path from the positive real axis to the line 
$\set{r e^{\ii\phi}: r>0}$,
for every fixed $\phi$ with $|\phi|<\frac\pi6$. Consequently, we
further have, for $x>0$ and $|\phi|<\frac\pi6$,
  \begin{equation}\label{g6}
	g(x)=\frac{3}{4\pi\ii}e^{(\nu-3/2)\ii\phi}\intoo
	\expx{- e^{-3\ii\phi/2}x r\qqcw}r^{\nu-5/2}\Hx\xpar{re^{\ii\phi}}\dd r.
  \end{equation}
The \rhs{} of \eqref{g6} is an analytic function of $x$ in the sector
$\set{x:|\arg x-3\phi/2|<\pi/2}$, which contains the positive real
axis; together, these thus define an analytic extension of 
$g(x)$ to the sector $|\arg x|<3\pi/4$ such that
\eqref{g6} holds whenever $|\arg x|<3\pi/4$,
$|\phi|<\frac\pi6$ and
$|\arg x-3\phi/2|<\pi/2$.

We next find the density of $X$ from $g$ by another Laplace inversion.
Assume first, for simplicity, that we already know that $X$ has a
continuous density $f$ on $(0,\infty)$.
Then $t\qqc X$ has the density $t\qqcw f(t\qqcw x)$, and thus (using
$t=x\qqa s$), for $x>0$,
\begin{equation}
  \label{a2a}
  \begin{split}
g(x)&=
\Gamma(\nu)\qw\intoo t^{\nu-1}e^{-t} t\qqcw f\bigpar{t\qqcw x}\dd t
\\
&=
\Gamma(\nu)\qw x^{2\nu/3-1}
\intoo e^{-x\qqa s} s^{\nu-5/2}f\bigpar{s\qqcw}\dd s.
  \end{split}
\end{equation}
Let 
\begin{equation}
  \label{phif}
\phif(s)\=s^{ \nu-5/2} f\bigpar{s\qqcw}.
\end{equation}
Then \eqref{a2a} can be written, with $x=y\qqc$,
\begin{equation}
  g\xpar{y\qqc}
=
\Gamma(\nu)\qw y^{\nu-3/2}\intoo e^{-y s} \phif(s)\dd s,
\qquad y>0.
\end{equation}
In other words, $\phif$ has the Laplace transform
\begin{equation}\label{tphif}
\tphif(y)
\=\intoo e^{-y s} \phif(s)\dd s
=
\Gamma(\nu) y^{3/2-\nu}  g\xpar{y\qqc},
\qquad y>0.
\end{equation}
Since this is finite for all $y>0$, the Laplace transform $\tphif$ is analytic
in the half-plane $\Re y>0$.
Hence, using our analytic exytension of $g$ to $|\arg z|<3\pi/4$,, 
\eqref{tphif} holds for all $y$ with
$\Re y>0$. 
Consequently, by standard Laplace inversion,
for every
$s>0$ and every
$\xi>0$ such that the integrals are absolutely 
(or even conditionally, see \refL{LLinv}) convergent,
\begin{equation}\label{phifinv}
  \begin{split}
  \phif(s)&=\frac{1}{2\pi\ii} \int_{\xi-\ii\infty}^{\xi+\ii\infty} 
e^{sy} \tphif(y) \dd y
=\frac{\Gamma(\nu)}{2\pi\ii} \int_{\xi-\ii\infty}^{\xi+\ii\infty} 
e^{sy}  y^{3/2-\nu}  g\xpar{y\qqc} \dd y.
  \end{split}
\end{equation}
We have for (mainly notational) simplicity assumed that $X$ has a
density. In general, we may replace the density function $f$ in
\eqref{a2a} and \eqref{phif} by a probability measure $\mu$ (with
suitable interpretations; we identify here absolutely continuous
measures and their densities as in the theory of distributions).
Then $\phif$ is a (positive) measure on $(0,\infty)$, and its Laplace
transform is still given by \eqref{tphif}.
The fact, proved below, that $\tphif$ is absolutely integrable on a
vertical line $\Re y=\xi$ implies by 
standard Fourier analysis 
that $\phif$ actually is the continuous function given by
\eqref{phifinv}, and thus the measure 
$\mu$ too is a continuous function; \ie, $X$ has a continuous density
$f$ as asserted, and 
\eqref{a2a} and \eqref{phif} hold.

We change variables in \eqref{phifinv}
to $\gt:=\arg y\in(-\pi/2,\pi/2)$, that is 
$y=\xi(1+\ii\tan\gt)=\xi\sec(\gt) e^{\ii\gt}$. 
We further 
express $g(y\qqc)$ by \eqref{g6} with $\phi=\gt/3$ (which satisfies
the conditions above for \eqref{g6}); 
this yields, assuming absolute convergence of
the double integral,
{\multlinegap=0pt
\begin{multline}\label{phifs}
\phif(s)
=\frac{3\Gamma(\nu)}{8\pi^2\ii}\int_{\gt=-\pi/2}^{\pi/2}\int_{r=0}^\infty
\exp\Bigpar{\xi s(1+\ii\tan\gt)-e^{\ii\gt}(\xi\sec\gt)\qqc r\qqcw}
\\
e^{(1-2\nu/3)\ii\gt}
\xi^{5/2-\nu}\xpar{\sec\gt}^{7/2-\nu} 
r^{\nu-5/2} \Hx\xpar{re^{\ii\gt/3}}
\dd r \dd\gt.
\end{multline}}
To verify absolute convergence of this double integral, take absolute
values inside the integral.
Since 
\begin{equation*}
 \Re\bigpar{e^{\ii\gt}(\xi\sec\gt)\qqc r\qqcw}
\ge \xi\qqc (\sec\gt)\qq r\qqcw,
\end{equation*}
the resulting integral is, using \eqref{Hxoox} and \eqref{Hx0},
for fixed $s$ and $\xi$ bounded by 
\begin{equation*}
\CC(s,\xi)
\int_{0}^{\infty}\int_{0}^\infty
e^{-\xi\qqc (\sec\gt)\qq r\qqcw}
(\sec\gt)^{7/2-\nu} 
r^{\nu-5/2} \min\bigpar{r^{-\nu},r^{-6}}
\dd r \dd \gt.
\end{equation*}
We split this double integral into the two parts: $0<\gt<1$ and
$1<\gt<\infty$. For $0<\gt<1$, $\sec\gt$ is bounded above and below,
and it is easy to see that the integral is finite. For $\gt>1$,
$\tan\gt<\sec\gt<2\tan\gt$, and with $t=\tan\gt$ we obtain at most 
\begin{equation*}
\CC
\int_{1}^{\infty}\int_{0}^\infty
e^{-\xi\qqc t\qq r\qqcw}
t^{3/2-\nu} 
r^{\nu-5/2} \min\bigpar{r^{-\nu},r^{-6}}
\dd r \dd t.
\end{equation*}
Substituting $t=r^3u$, we find that this is at most
\begin{equation*}
\CCx
\int_{0}^{\infty}
e^{-\xi\qqc u\qq}
u^{3/2-\nu} \dd u
\int_{0}^\infty
r^{5-2\nu} \min\bigpar{r^{-\nu},r^{-6}}
\dd r 
<\infty.
\end{equation*}
This verifies absolute convergence of the double integral in
\eqref{phifs} for every $\xi>0$, which implies absolute convergence of
the integrals in \eqref{phifinv}.
Consequently, 
\eqref{phifinv} and
\eqref{phifs} are valid for every $s>0$ and $\xi>0$.
We now put $s=x\qqaw$ in \eqref{phifs} and obtain by \eqref{phif} the
sought result \eqref{ff}.

\begin{remark}
We have chosen $\phi=\gt/3$, which leads to \eqref{ff} and, see
\refR{RTol}, the formulas by \citet{Tol:br,Tol:bm,Tol:brp}.
Other choices of $\phi$ are possible and lead to variations of the inversion
formula \eqref{ff}.
In particular, it may be noted that we may take $\phi=0$ for, say,
$|\gt|<1$; this yields a formula that, apart from a small
contribution
for $|\gt|>\pi/4$, involves $\Hx(x)$ for real $x$ only.
However, we do not find that this or any other variation of \eqref{ff}
simplifies the application of the saddle method, and we leave these
versions to the interested reader.
\end{remark}

\section{Moment asymptotics}\label{Smoments}

Suppose that $X$ is a positive random variable with a density function
$f$ satisfying \eqref{fas}.
Then, as \rtoo, using Stirling's formula,
\begin{equation}\label{m1}
  \begin{split}
\E X^r	
&\sim 
\intoo a x^{r+\ga}e^{-bx^2}\dd x
\\&
=
\frac a2\intoo  y^{(r+\ga+1)/2-1}e^{-by}\dd y
\\&
=
\frac a2b^{-(r+\ga+1)/2}\Gamma\parfrac{r+\ga+1}2
\\&
\sim
\frac a2b^{-(r+\ga+1)/2}\parfrac r2^{(\ga+1)/2}\Gamma\parfrac{r}2
\\&
=
a\sqrt{\pi}(2b)^{-(\ga+1)/2} r^{\ga/2}\parfrac{r}{2eb}^{r/2}.
  \end{split}
\end{equation}
(It is easily seen, by an integration by parts, that the same result
follows from the weaker assumption \eqref{Fas}.)

For the Brownian areas studied in this paper, Theorems
\ref{Tex}--\ref{Tbmp} thus imply the following.
\begin{corollary}
As \ntoo,
\begin{align}
\E \cBex^n&\sim 3\sqrt2\, n \parfrac{n}{12e}^{n/2} ,\label{exmom}
\\  
 \E \cBbr^n&\sim \sqrt2\ \parfrac{n}{12e}^{n/2}, \label{brmom}
\\  
 \E \cBbm^n&\sim \sqrt2\ \parfrac{n}{3e}^{n/2}, \label{bmmom}
\\  
 \E \cBme^n&\sim \sqrt{3\pi}n\qq \parfrac{n}{3e}^{n/2}, \label{memom}
\\  
 \E \cBdm^n&\sim 2\sqrt2\ \parfrac{n}{3e}^{n/2}, \label{dmmom}
\\
 \E \cBbrp^n&\sim \frac1{\sqrt2} \parfrac{n}{12e}^{n/2}, \label{brpmom}
\\
\E \cBbmp^n&\sim \frac1{\sqrt2} \parfrac{n}{3e}^{n/2}. \label{bmpmom}
\end{align}  
\end{corollary}
Most of these results 
have been found earlier:
\eqref{exmom} by \citet{Takacs:Bernoulli},
\eqref{brmom} by \citet{Takacs:walk} and \citet{Tol:br},
\eqref{bmmom} by \citet{Takacs:BM} and \citet{Tol:bm},
\eqref{memom} by \citet{Takacs:meander},
\eqref{dmmom} by \citet{SJ201},
\eqref{brpmom} by \citet{Tol:brp};
\Takacs{} used recursion formulas derived by other methods, while
Tolmatz used the method followed here.
Note that, as remarked by \citet{Tol:brp}, 
$\E \cBbrp^n\sim \half \E \cBbr^n$
and similarly
$\E \cBbmp^n\sim \half \E \cBbm^n$,
\cf{} \refR{R+}.

In the opposite direction, we do not know any way to get precise
asymptotics of the form \eqref{fas} or \eqref{Fas} from moment
asymptotics, but, as observed by \citet{CsorgoShiYor},
the much weaker
estimate \eqref{sofie} and its analogue for other Brownian areas
can be obtained by the following special case of results
by Davies \cite{Davies} and Kasahara \cite{Kasahara}.
(See \cite[Theorem 4.5]{SJ161} for a more general version with an
arbitrary power $x^p$ instead of $x^2$ in the exponent.)

\begin{proposition}
  If $X$ is a positive random variable and $b>0$, then the following
  are equivalent:
  \begin{align*}
-\ln\P(X>x)&\sim b x^2, && \xtoo,
\\
\bigpar{\E X^n}^{1/n} &\sim \sqrt{\frac{n}{2eb}}, && \ntoo,
\\
\ln\bigpar{\E e^{tX}} &\sim \frac1{4b}t^2, && \ttoo.
  \end{align*}
\end{proposition}

Returning to \eqref{m1}, we obtain in the same way more precise
asymptotics for the moments if we are given an asymptotic series for
$f$ or $\P(X>x)$. For simplicity, we consider only the next term, but
the calculations can be extended to an asymptotic expansion with any
number of terms. Thus, suppose that, as for the Brownian areas, 
\eqref{fas} is sharpened to \eqref{fasx} with $N\ge2$. Then, also
using further terms in Stirling's formula,
\begin{equation*}
  \begin{split}
\E X^n	
&=
\frac {a_0}2b^{-(n+\ga+1)/2}\Gamma\parfrac{n+\ga+1}2
+
\frac {a_2}2b^{-(n+\ga-1)/2}\Gamma\parfrac{n+\ga-1}2
\\
&\hskip16em
+O\lrpar{b^{-n/2}\Gamma\parfrac{n+\ga-3}2}
\\
&=
\frac {1}2b^{-(n+\ga+1)/2}\Gamma\parfrac{n+\ga+1}2
\lrpar{a_0+a_2b\frac2{n}+O(n\qww)}
\\&
=
\sqrt{2\pi}(2b)^{-(\ga+1)/2} n^{\ga/2}\parfrac{n}{2eb}^{n/2}
\\ &\hskip8em
\cdot
\lrpar{a_0+\lrpar{a_0\frac{\ga^2-1}{4}
+\frac{a_0}6
+{2a_2b}}n\qw
+O(n\qww)}.
  \end{split}
\end{equation*}
In particular, for the Brownian excursion, where by \refT{Tex}
\eqref{fasx} holds with $\ga=2$, $b=6$, $a_0=72\sqrt{6/\pi}$
and $a_2=-8\sqrt{6/\pi}$,
\begin{equation}\label{bex1}
  \begin{split}
\E \cBex^n	
&
=\frac1{2\sqrt2} \parfrac{n}{12e}^{n/2}n
\lrpar{12+\frac{9+2-16}n
+O(n\qww)}
\\&
=3\sqrtt \parfrac{n}{12e}^{n/2}n
\lrpar{1-\frac{5}{12n}
+O(n\qww)}.
  \end{split}
\end{equation}
If we, following \citet{Takacs:Bernoulli}, introduce $K_n$ defined by
\begin{equation*}
  \E \cBex^n =\frac{4\sqrt{\pi}\,2^{-n/2}n!}{\Gamma((3n-1)/2)}K_n,
\end{equation*}
further applications of Stirling's formula shows that \eqref{bex1} is
equivalent to 
\begin{equation}\label{Kex}
K_n ={(2\pi)}\qqw n\qqw\parfrac{3n}{4e}^{n} \lrpar{1-\frac7{36n}+O(n\qww)}.
\end{equation}
Again, the leading term is given by \citet{Takacs:Bernoulli}, in the
equivalent form
\begin{equation}\label{Ksim}
K_n \sim \frac1{2\pi}\parfrac{3}{4}^{n} (n-1)!
\qquad \text{as } \ntoo.
\end{equation}
\citet{Takacs:Bernoulli} further gave the recursion formula
(with $K_0=-1/2$)
\begin{equation}\label{K-rec}
K_n
=\frac{3n-4}{4}K_{n-1}
  +\sum_{j=1}^{n-1}K_jK_{n-j},
\qquad n\ge1,
\end{equation}
It is easy to obtain from \eqref{Ksim} and \eqref{K-rec} the refined
asymptotics 
\begin{equation}\label{Kexx}
K_n = \frac1{2\pi}\parfrac{3}{4}^{n} (n-1)!
\lrpar{1-\frac5{18n}+O(n\qww)},
\end{equation}
which is equivalent to \eqref{Kex}.
and, by recursion, \eqref{Kexx} can be extended to an asymptotic expansion of
arbitrary length.
(Another method to obtain an asymptotic expansion of $K_n$ is given by
\citet{KearneyMM}.) 
Hence \eqref{bex1} (also with further terms) can, alternatively, be
derived from \eqref{Ksim} and \eqref{K-rec} by straightforward
calculations. However, as said above, we do not know any way to derive
\refT{Tex}
from this.
(Nevertheless, the calculations above serve as a check of the
coefficients in \refT{Tex}.)


\begin{acks}
We thank Philippe Flajolet for interesting discussions.
This research was partly done during a visit by SJ to the University of
Cambridge, partly funded by Trinity College. 
\end{acks}

\appendix

\section{Proof that $\sqrt{z}\Ai(z)-\Ai'(z)$ has no zeros}\label{Azeros}

The double Laplace transforms for the positive part areas $\cB\brp$ and
$\cB\bmp$ have both the denominator 
$\sqrt{z}\Ai(z)-\Ai'(z)$, 
and it is important that this function has no zeros, whence $\Hx\brp$
and $\Hx\bmp$ are analytic in \slitplane.
This was proved by \citet{Tol:brp} (for the same reason), but
we give here an alternative proof that does not need the careful
numerical integration done by Tolmatz. 
(Our proof is, like Tolmatz', based on the argument principle.) 

\begin{lemma}[\citet{Tol:brp}]
  The function $\sqrt{z}\Ai(z)-\Ai'(z)$ is non-zero for all $z=r
  e^{\ii\gt}$ with $r\ge0$ and $|\gt|\le\pi$.
\end{lemma}

\begin{proof}
We use the notations
\begin{align*}
\zeta(z)&\=\tfrac23 z\qqc,
\\
  f(z)&\=\sqrt{\pi}\lrpar{\sqrt{z}\Ai(z)-\Ai'(z)},
\\
g(z)&\=e^{\zeta(z)}f(z).
\end{align*}
Note that these functions are analytic in \slitplane{} and extend 
continuously to $(-\infty,0]$ from each side, so we can regard them as
  continuous functions of $re^{\ii\gt}$ with $r\ge0$ and $-\pi\le
  r\le\pi$, where we regard the two sides $re^{\pm\ii\pi}=-r\pm\ii0$
of the negative real axis as different. (The reader that dislikes this
can reformulate the proof and study $z\Ai(z^2)-\Ai'(z^2)$ for $\Re
z\ge0$; this avoids the ambiguities of square roots.)

  We will use the argument principle on $g(z)$ and the contour
  $\gam_R$ consisting of the interval from $0$ to $-R-\ii0$ along the
  lower side of the negative real axis, the circle $Re^{\ii\gt}$ for
  $-\pi\le\gt\le\pi$ and the interval from $-R+\ii0$ back to 0, where $R$
  is a large real number.

First, fix a small $\gd>0$. By \eqref{Aioo} and \eqref{Aiioo}, as
$|z|\to\infty$, 
\begin{equation}
  \label{f1}
f(z)\sim z\qqqq e^{-\zeta(z)},
\qquad \aaz\le\pi-\gd.
\end{equation}

Next, assume $0<\arg z<2\pi/3-\gd$. Note that then
$\arg(-z)=\arg(z)-\pi\in(-\pi,-\pi/3-\gd)$ and thus 
\begin{equation}
  \label{zeta-}
(-z)\qq=-\ii z\qq,\qquad
\zeta(-z)=e^{-(3/2)\ii\pi}\zeta(z)=\ii\zeta(z).
\end{equation}
Furthermore, we have 
as $|z|\to\infty$ with $\aaz<2\pi/3-\gd$
the expansions
\cite[10.4.60,10.4.62]{AS}
\begin{align*}
\Ai(-z)&= 
{\pi^{-1/2}}z^{-1/4}
\Bigpar{\sin\bigpar{\zeta(z)+\frac\pi4}\bigpar{1+O(\zeta\qww)}
-\cos\bigpar{\zeta(z)+\frac\pi4}\cdot O(\zeta\qw)
}
\\
\Ai'(-z)&= 
{\pi^{-1/2}}z^{1/4}
\Bigpar{-\cos\bigpar{\zeta(z)+\frac\pi4}\bigpar{1+O(\zeta\qww)}
+\cos\bigpar{\zeta(z)+\frac\pi4}\cdot O(\zeta\qw)
}
\end{align*}
and thus
\begin{align*}
f(-z)&= 
z^{1/4}
\Bigpar{\cos\bigpar{\zeta(z)+\frac\pi4}\bigpar{1+O(\zeta\qw)}
-\ii\sin\bigpar{\zeta(z)+\frac\pi4}\bigpar{1+O(\zeta\qw)}
}.
\end{align*}
In the range $\arg z\in(0,2\pi/3-\gd)$, further $\Im\zeta(z)>0$, and
thus by Euler's formulas
\begin{equation*}
  \lrabs{\cos\bigpar{\zeta(z)+\frac\pi4}}
+
\lrabs{\sin\bigpar{\zeta(z)+\frac\pi4}}
\le \lrabs{e^{\ii\zeta(z)}} + \lrabs{e^{-\ii\zeta(z)}}
\le 2 e^{\Im\zeta(z)}
=
2\lrabs{e^{-\ii\zeta(z)}}.
\end{equation*}
Hence, as $|z|\to\infty$ with
$\arg z\in(0,2\pi/3-\gd)$, 
using \eqref{zeta-}.
\begin{align*}
f(-z)&=
z^{1/4}
e^{-\ii(\zeta(z)+\xfrac\pi4)}\bigpar{1+O(\zeta\qw)}
\sim
(-z)\qqqq e^{-\zeta(-z)}.
\end{align*}
Consequently, \eqref{f1} holds as \ztoo{} with 
$-\pi<\arg z<-\pi/3-\gd$ too. Since $f(\overline z)=\overline{f(z)}$, it
holds for $\pi/3+\gd<\arg z<\pi$ too, and combining the three ranges,
we see that as \ztoo, for all $\aaz<\pi$,
\begin{equation}
  f(z)\sim z\qqqq e^{-\zeta(z)},
\end{equation}
and thus
\begin{equation}\label{gr}
  g(z)= z\qqqq\bigpar{1+o(1)}.
\end{equation}

Consider now $f(z)$ on the lower side of the negative real axis, \ie{}
for $z=r e^{-\ii\pi}=-r-\ii0$, $r\ge0$.
Note that then $\Ai(z)$ and $\Ai'(z)$ are real and $z\qq$ purely
imaginary.
Since $\Ai$ and $\Ai'$ have no common zeros, and $\Ai'(0)\neq0$,
$f(-r-\ii0)\neq0$. Moreover, $f(0)>0$, and as $r$ grows from 0 to
$\infty$,
$f(-r-\ii0)$ is real at $r=0$ and at the zeros $-r=a_k$ of $\Ai$, and
imaginary at the zeros $-r=a_k'$ of $\Ai'$. 
Consider continuous determinations of $\arg f(z)$ and $\arg g(z)$
along the neagtive
real axis,
starting with $\arg f(0)=\arg g(z)=0$.
It is easily seen that $\arg f(z)$ then is $-\pi/2$ for $z=a_1'$,
$-\pi$ for $z=a_1$, and so on, with $\arg f(a_k-\ii0)=-k\pi$.
Furthermore, for $z=-r-\ii0$,
\begin{equation*}
  \arg g(z)=\arg f(z)+\Im \zeta(z)=\arg f(z)+\tfrac23\Im z\qqc
=\arg f(z)+\tfrac23 r\qqc.
\end{equation*}
In particular, using the asymptotic formula \cite[10.4.94]{AS} for the
Airy zeros $a_k$,
\begin{equation}\label{g0-}
  \begin{split}
  \arg g(a_k)
&=-\pi k+\tfrac23 |a_k|\qqc
=-\pi k+\frac23\,\frac{3\pi(4k-1)}{8}\bigpar{1+O(k\qww)}
\\&
=-\frac{\pi}{4}+O(k\qw).	
  \end{split}
\end{equation}

Consider now a continuous determination of
$\arg g(z)$ along the contour $\gam_R$, with $R=|a_k|$
for a large $k$. On the part from $0$ to $-R-\ii0$, the argument decreases
by $-\pi/4+O(k\qw)$ by \eqref{g0-}, and on the half-circle from
$-R-\ii0$ to $R$, it increases by \eqref{gr} by $\pi/4+o(1)$, so the
total change from 0 to $R$
is $o(1)$, \ie, tends to 0 as $k\to\infty$. Since furthermore
$g(R)>0$, the change is a multiple of $2\pi$, and thus exactly 0 for
large $k$. By symmetry, the change of the argument on the remaining
half of $\gam_R$ is the same, so the total change along $\gam_R$ is 0,
which proves that $g(z)$ has no zero inside $\gam_R$ for $R=|a_k|$
with $k$ large. Letting $k\to\infty$, we see that $g(z)$ has no zeros.
\end{proof}

\newcommand\AAP{\emph{Adv. Appl. Probab.} }
\newcommand\JAP{\emph{J. Appl. Probab.} }
\newcommand\JAMS{\emph{J. \AMS} }
\newcommand\MAMS{\emph{Memoirs \AMS} }
\newcommand\PAMS{\emph{Proc. \AMS} }
\newcommand\TAMS{\emph{Trans. \AMS} }
\newcommand\AnnMS{\emph{Ann. Math. Statist.} }
\newcommand\AnnPr{\emph{Ann. Probab.} }
\newcommand\CPC{\emph{Combin. Probab. Comput.} }
\newcommand\JMAA{\emph{J. Math. Anal. Appl.} }
\newcommand\RSA{\emph{Random Struct. Alg.} }
\newcommand\ZW{\emph{Z. Wahrsch. Verw. Gebiete} }
\newcommand\DMTCS{\jour{Discr. Math. Theor. Comput. Sci.} }

\newcommand\AMS{Amer. Math. Soc.}
\newcommand\Springer{Springer-Verlag}
\newcommand\Wiley{Wiley}

\newcommand\vol{\textbf}
\newcommand\jour{\emph}
\newcommand\book{\emph}
\newcommand\inbook{\emph}
\def\no#1#2,{\unskip#2, no. #1,} 

\newcommand\webcite[1]{\hfil\penalty0\texttt{\def~{\~{}}#1}\hfill\hfill}
\newcommand\webcitesvante{\webcite{http://www.math.uu.se/\~{}svante/papers/}}
\newcommand\arxiv[2][\relax]{\webcite{arXiv:#2 \testtom{#1}}}
\newcommand\testtom[1]{%
\def\xxa{#1\relax}\def\xxb{\relax}\ifx\xxa\xxb \else [#1]\fi}

\def\nobibitem#1\par{}


\begin{thebibliography}{99}


\bibitem{AS}
M. Abramowitz \& I. A. Stegun, eds.,
\book{Handbook of Mathematical Functions}.
Dover, New York, 1972.

\nobibitem{Aldous}
D. Aldous,
Brownian excursions, critical random graphs and the multiplicative
coalescent. 
\emph{Ann. Probab.}
\vol{25} (1997),
812--854.

\nobibitem{BagDm}
G. N. Bagaev \& E. F. Dmitriev, 
Enumeration of connected labeled bipartite graphs. (Russian.) 
\jour{Doklady Akad. Nauk BSSR} \vol{28} (1984),
1061--1063.


\nobibitem{BenderCanfieldMcKay}
E. A. Bender, E. R. Canfield \&  B. D. McKay,
Asymptotic properties of labeled connected graphs.
\RSA \vol3\no2 (1992),  183--202.

\bibitem[Bleistein, Handelsman 1986]{BH86}
N. Bleistein  \& R. A. Handelsman,
\book{Asymptotic expansions of Integrals.} 
Dover Publications, 1986.





\nobibitem{Cifarelli}
D. M. Cifarelli, 
Contributi intorno ad un test per l'omogeneit\`a tra due campioni.  
\jour{Giorn. Econom. Ann. Econom. (N.S.)}  \vol{34}\no{3--4}  (1975), 
233--249.

\bibitem[Cs\"org{\H o}, Shi and Yor(1999)]{CsorgoShiYor}
M. Cs\"org{\H o}, Z. Shi \& M. Yor,
Some asymptotic properties of the local time of the uniform empirical
process.  
\jour{Bernoulli}  \vol{5}  (1999),  no. 6, 1035--1058. 

\bibitem{Davies}
L. Davies. 
Tail probabilities for positive random variables with entire
characteristic functions of very regular growth. 
\jour{Z. Angew. Math. Mech.} \vol{56} (1976), no. 3, T334--T336.

\bibitem{Fatalov}
V. R. Fatalov,  
Asymptotics of large deviations of Gaussian processes
of Wiener type for $L\sp p$-functionals, $p>0$, and the hypergeometric
function. (Russian)  
\jour{Mat. Sb.}  \vol{194}  (2003),  no. 3, 61--82;
English translation \jour{Sb. Math.}  \vol{194}  (2003),  no. 3-4, 369--390.

\bibitem{SJ197}
J. A. Fill \& S. Janson,
Precise logarithmic asymptotics for the right tails of some limit
  random variables for random trees.
Preprint, 2007.
\arxiv[math.PR]{math/0701259v1}

\bibitem{FL:Airy}
P. Flajolet \& G. Louchard,
Analytic variations on the Airy distribution. 
\jour{Algorithmica} \vol{31} (2001),
361--377.

\nobibitem{FPV} 
P. Flajolet, P. Poblete \& A. Viola,
On the analysis of linear probing hashing.
\jour{Algorithmica} \vol{22}\no4 (1998), 
490--515.

\nobibitem{Groeneboom}
P. Groeneboom, 
Brownian motion with a parabolic drift and Airy functions.  
\jour{Probab. Theory Related Fields}  \vol{81}  (1989),  no. 1, 79--109. 

\nobibitem{SJGauss}
S. Janson,
\book{Gaussian Hilbert Spaces}.
Cambridge Univ. Press, Cambridge, 1997.

\nobibitem{SJ146}
S. Janson,
The Wiener index of simply generated random trees.
\RSA \vol{22}\no4 (2003), 337--358. 

\bibitem{SJ161}
S. Janson \& P. Chassaing,
The center of mass of the ISE and the Wiener index of trees.
\emph{Electronic Comm. Probab.} \vol{9} (2004), paper 20, 178--187.

\bibitem[Janson(2007)]{SJ201}
S. Janson.
Brownian excursion area,
Wright's constants in graph enumeration, 
and other Brownian areas.
\jour{Probability Surveys}
\vol{4} (2007), 80--145.

\nobibitem{JohnsonKilleen}
B. McK. Johnson \& T. Killeen,  
An explicit formula for the C.D.F. of the $L\sb{1}$ norm of the
Brownian bridge.  
\jour{Ann. Probab.}  \vol{11}   (1983),  no. 3, 807--808. 

\bibitem{Kac46}
M. Kac,
On the average of a certain Wiener functional and a related limit
theorem in calculus of probability.  
\jour{Trans. Amer. Math. Soc.}  \vol{59} (1946), 401--414. 

\nobibitem{Kac49}
M. Kac,
On distributions of certain Wiener functionals.  
\jour{Trans. Amer. Math. Soc.}  \vol{65} (1949), 1--13. 

\nobibitem{Kac51}
M. Kac,
On some connections between probability theory and differential and
integral equations. 
\inbook{Proceedings of the Second Berkeley Symposium on Mathematical
  Statistics and Probability, 1950},  
University of California Press, Berkeley and Los Angeles, 1951,
pp. 189--215.

\bibitem{Kasahara}
Y. Kasahara. 
Tauberian theorems of exponential type.  
\jour{J. Math. Kyoto Univ.}  \vol{18}  (1978), no. 2, 209--219.

\bibitem[Kearney, Majumdar and Martin(2007+)]{KearneyMM}
M. J. Kearney, S. N. Majumdar \& R. J. Martin,
The first-passage area for drifted Brownian motion and the moments of
the Airy distribution.
Preprint, 2007.
\arxiv[cond-mat.stat-mech]{0706.2038v1}

\nobibitem{Lebedev}
N. N. Lebedev,
\emph{Special Functions and their Applications}.
Dover, New York, 1972. (Translated from Russian.)



\bibitem{Lou:kac}
G. Louchard,
Kac's formula, L\'evy's local time and Brownian excursion. 
\emph{J. Appl. Probab.} \vol{21}\no3 (1984),
479--499.

\bibitem{Lou:ex}
G. Louchard,
The Brownian excursion area: a numerical analysis. 
\emph{Comput. Math. Appl.} \vol{10}\no6 (1984),
413--417.
Erratum:
\emph{Comput. Math. Appl. Part A} \vol{12}\no3 (1986), 
375.

\bibitem[Majumdar and Comtet(2005)]{MC:Airy} 
S. N. Majumdar \& A. Comtet,  
Airy distribution function: from the area under a Brownian excursion
to the maximal height of fluctuating interfaces.  
\jour{J. Stat. Phys.}  \vol{119}  (2005),  no. 3-4, 777--826.

\bibitem[Milnor(1963)]{Milnor}
J. Milnor,
\book{Morse theory.} 
Princeton University Press, Princeton, N.J., 1963.

\bibitem{PermanW}
M. Perman \& J. A. Wellner,
On the distribution of Brownian areas.  
\jour{Ann. Appl. Probab.}  \vol6  (1996),  no. 4, 1091--1111.

\bibitem[Revuz and Yor(1999)]{RY}
D. Revuz \& M. Yor,
\book{Continuous martingales and Brownian motion}. 
3rd edition. 
\Springer, Berlin, 1999. 


\nobibitem{Rice}
S. O. Rice, 
The integral of the absolute value of the pinned Wiener
process---calculation of its probability density by numerical
integration.  
\jour{Ann. Probab.}  \vol{10}  (1982), no. 1, 240--243. 

\nobibitem{Richard}
C. Richard,
On $q$-functional equations and excursion moments.
Preprint, 2005.
\arxiv{math.CO/0503198}

\nobibitem{Shepp}
L. A. Shepp, 
On the integral of the absolute value of the pinned Wiener process.  
\jour{Ann. Probab.}  \vol{10} \no1  (1982), 234--239.

\nobibitem{Spencer}
J. Spencer, 
Enumerating graphs and Brownian motion.  
\jour{Comm. Pure Appl. Math.}  \vol{50}\no3  (1997),  291--294. 

\bibitem[\Takacs(1991)]{Takacs:Bernoulli}
L. \Takacs,
A Bernoulli excursion and its various applications. 
\emph{Adv. Appl. Probab.} \vol{23}\no3 (1991), 
557--585.

\nobibitem{Takacs:rail}
L. \Takacs, 
On a probability problem connected with railway traffic.
\emph{J. Appl. Math. Stochastic Anal.} \vol4 \no1 (1991),
1--27. 

\nobibitem{Takacs:cond}
L. \Takacs, 
Conditional limit theorems for branching processes. 
\emph{J. Appl. Math. Stochastic Anal.} \vol4 \no4 (1991), 263--292. 

\bibitem[\Takacs(1992)]{Takacs:walk}
L. \Takacs, 
Random walk processes and their applications to order statistics.
\emph{Ann. Appl. Probab.} \vol2 \no2 (1992), 435--459.

\nobibitem{Takacs:binary}
L. \Takacs, 
On the total heights of random rooted binary trees. 
\emph{J. Combin. Theory Ser. B} \vol{61}\no2 (1994), 155--166.

\bibitem[\Takacs(1993)]{Takacs:BM}
L. \Takacs, 
On the distribution of the integral of the absolute value of the
Brownian motion.  
\emph{Ann. Appl. Probab.} \vol3\no1 (1993), 186--197.

\bibitem[\Takacs(1995)]{Takacs:meander}
L. \Takacs, 
Limit distributions for the Bernoulli meander. 
\jour{J. Appl. Probab.} \vol{32} (1995), no. 2, 375--395. 

\bibitem[Tolmatz(2000)]{Tol:br}
L. Tolmatz, 
Asymptotics of the distribution of the integral of the
absolute value of the Brownian bridge for large
arguments. 
\jour{Ann. Probab.} \vol{28} (2000), no. 1, 132--139.

\bibitem[Tolmatz(2003)]{Tol:bm}
L. Tolmatz, 
The saddle point method for the integral of the
absolute value of the Brownian motion. 
\inbook{Discrete random walks (Paris,
2003),  
Discrete Math. Theor. Comput. Sci. Proc.} \vol{AC},  Nancy, 2003, 
pp. 309--324.

\bibitem[Tolmatz(2005)]{Tol:brp}
L. Tolmatz, 
Asymptotics of the distribution of the integral of the positive part
of the Brownian bridge for large arguments.  
\jour{J. Math. Anal. Appl.}  \vol{304}
(2005),  no. 2, 668--682. 


\nobibitem{Vervaat} 
W. Vervaat, 
A relation between Brownian bridge and Brownian excursion.
\jour{Ann. Probab.} \vol7\no1 (1979),  143--149.

\nobibitem{Voblyi}
V. A. \Voblyi, 
O koeffitsientakh Ra{\u\i}ta i Stepanova-Ra{\u\i}ta.
\jour{Matematicheskie Zametki} \vol{42} (1987), 854--862. 
English translation:
Wright and Stepanov-Wright coefficients. 
\jour{Mathematical Notes} \vol{42} (1987), 969--974.

\nobibitem{Watson}
G. S. Watson, 
Goodness-of-fit tests on a circle.  
\jour{Biometrika}  \vol{48} (1961), 109--114.

\nobibitem{Wright}
E. M. Wright,
The number of connected sparsely edged graphs. 
\emph{J. Graph Th.} \vol{1} (1977),
317--330.

\nobibitem{WrightIII}
E. M. Wright,
The number of connected sparsely edged graphs. III.
Asymptotic results.
\emph{J. Graph Th.} \vol{4} (1980),
393--407.



\end{thebibliography}
\end{document}